\documentclass[10pt,twocolumn,twoside]{IEEEtran} 

\usepackage{amssymb,theorem,amsmath,bbm}
\usepackage[usenames,dvipsnames]{color}
\usepackage{graphicx,dblfloatfix_ieee}
\usepackage{mathrsfs,mathtools}
\usepackage{subfigure,url}

\newtheorem{theorem}{Theorem}[section]
\newtheorem{lemma}[theorem]{Lemma}

\newtheorem{corollary}[theorem]{Corollary}
\newtheorem{proposition}[theorem]{Proposition}
\newtheorem{definition}[theorem]{Definition}

\newcommand{\reals}{\mathbb{R}} 
\newcommand{\realsnonnegative}{\mathbb{R}_{\ge 0}} 
\newcommand{\real}{\mathbb{R}}
\newcommand{\ball}{\mathbb{B}}
\newcommand{\argmax}{\operatorname{argmax}}

\newcommand{\graph}{\mathcal{G}}

\newcommand{\maxerrors}{\mathbbm{M}} 
 
\newcommand{\vertices}{\mathcal{V}}
\newcommand{\weights}{W} 
\newcommand{\edges}{\mathcal{E}}
\newcommand{\matching}{M} 
\newcommand{\allocation}{\alpha}
 
\newcommand{\neigh}{\mathcal{N}}
\newcommand{\weight}{w} 
\newcommand{\slack}{s}
\newcommand{\besta}[2]{\texttt{ba}_{{#1} \setminus {#2}}}
\newcommand{\nextbest}[2]{\texttt{bn}_{{#1} \setminus {#2}}}
\newcommand{\convexhull}{\operatorname{co}}

\newcommand{\partner}{\mathcal{P}}
\newcommand{\nomflow}{f}

\newcommand{\oprocendsymbol}{\hbox{$\bullet$}}
\newcommand{\oprocend}{\relax\ifmmode\else\unskip\hfill\fi\oprocendsymbol}
\def\eqoprocend{\tag*{$\bullet$}}
\newcommand{\until}[1]{\{1,\dots,#1\}}

\newcommand{\BB}{\mathcal{B}}

\newcommand{\longthmtitle}[1]{\mbox{}\textup{\bf {(#1).}}}
\newcommand{\NormInf}[1]{\|#1\|_{\infty}}

\newcommand{\setdef}[2]{\{#1 \; |\; #2\}}

\newcounter{MYtempeqncnt}

\newcommand{\myclearpage}{\clearpage}
\renewcommand{\myclearpage}{}

\begin{document}


\title{Distributed bargaining 
  in dyadic-exchange networks}

\author{Dean Richert \qquad Jorge Cort\'es \thanks{A preliminary
    version of this paper appeared in the 2013 American Control
    Conference. The authors are with the Department of Mechanical and
    Aerospace Engineering, University of California, San Diego, CA
    92093, USA, {\tt\small \{drichert,cortes\}@ucsd.edu}}}

\maketitle

\begin{abstract}
  This paper considers dyadic-exchange networks in which individual
  agents autonomously form coalitions of size two and agree on how to
  split a transferable utility. Valid results for this game include
  stable (if agents have no unilateral incentive to deviate), balanced
  (if matched agents obtain similar benefits from collaborating), or
  Nash (both stable and balanced) outcomes.  We design
  provably-correct continuous-time algorithms to find each of these
  classes of outcomes in a distributed way.  Our algorithmic design to
  find Nash bargaining solutions builds on the other two algorithms by
  having the dynamics for finding stable outcomes feeding into the one
  for finding balanced ones.  Our technical approach to establish
  convergence and robustness combines notions and tools from
  optimization, graph theory, nonsmooth analysis, and Lyapunov
  stability theory and provides a useful framework for further
  extensions.  We illustrate our results in a wireless communication
  scenario where single-antenna devices have the possibility of
  working as 2-antenna virtual devices to improve channel capacity.
\end{abstract}

\section{Introduction}\label{sec:introduction}

Networked systems, characterized by distributed interactions among
multiple components and across several layers, are pervasive in modern
engineering problems and also model various biological, economic, and
sociological processes. Our motivation in this work is driven by
resource-constrained networks where collaboration between subsystems
gives rise to a more efficient use of the resources. To this end, we
view each subsystem as a player in a coalitional game where
neighboring players seek to form a match (i.e., a coalition of size
two) and split the corresponding transferable utility between the
matched agents.  Our aim is to synthesize distributed bargaining
mechanisms that agents can employ to decide with whom to collaborate
with and how to allocate the utility.  Our interest in distributed
strategies is motivated by the inherent limitations posed by the
network structure, privacy concerns, and scalability and robustness
considerations for implementation.

\emph{Literature review.}  Examples of networked systems where
performance benefits arise from agents cooperating with each other
towards the achievement of a common goal are by now pervasive, see
e.g.~\cite{WR-RWB:08,FB-JC-SM:08cor,MM-ME:10} and references therein.
Of particular relevance to this paper are scenarios where individual
agents have the ability to carry out their objectives satisfactorily
by themselves, but performance can be improved by collaborating with
others. Examples are also numerous and include resource allocation in
communication
networks~\cite{WS-ZH-ZD-MH-TB:09,SHA-KL-VCML:07,ZZ-SJ-CH-MG-QP:08},
formation creation in networks of UAVs~\cite{DR-JC:13-auto}, network
security~\cite{VMP-MZ-CE-AJ-GP:14}, mobile robot
coordination~\cite{MSS-KHJ-DMS:12}, large-scale data
processing~\cite{CJR-RN-MS:06}, 
and applications in sociology~\cite{TC-SJ-MK-JT:10} and
economics~\cite{AER:84}.  Bargaining problems of the type considered
here are posed on dyadic-exchange networks, so called because agents
can match with at most one other agent~\cite{KSC-RME:78}. Bipartite
matching and assignment problems~\cite{DPB:98} are special cases of
the dyadic-exchange network. Nash bargaining outcomes are an extension
to multi-player games of the classical two-player Nash bargaining
solution~\cite{JN:50-econometrica}. The
works~\cite{JK-ET:08,MB-MH-NI-HM:10} develop centralized methods for
finding such outcomes. In terms of distributed implementations, the
work~\cite{YA-BB-LC-ND-YP:09} provides a discrete-time dynamics that,
given a matching, converges to balanced allocations
and~\cite{MB-CB-JC-YK-AM:14} provides a discrete-time dynamics that
converges to Nash outcomes (without considering dynamics separately
for stable or balanced outcomes). With respect to these works, an
important novelty of the present manuscript is the dynamics and
control perspective on this class of problems, which allows us to
develop a principled technical approach to the study of asymptotic
convergence and robustness. Another area of connection with the
literature is the body of research on distributed algorithms for
solving linear programs~\cite{DPB-JNT:97,MB-GN-FB-FA:12}, including
the authors' previous work~\cite{DR-JC:13-tac}. Our algorithmic design
to find stable outcomes builds on the distributed algorithm proposed
in~\cite{DR-JC:13-tac} because of its mild requirements for
guaranteeing convergence and its robustness properties against
disturbances.

\emph{Statement of contributions.}  We consider dyadic-exchange
networks where individual agents bargain with one another about whom
to match with and how to allocate the transferable utility associated
to a matching. For this scenario, the type of outcomes we are
interested in are Nash bargaining solutions, which combine the notion
of stability and fairness. A stable outcome is one where none of the
agents benefit by unilaterally deviating from their match. A balanced
outcome is one where matched agents benefit equally from the
match, where benefit is defined as the next-best allocation that an
agent could achieve by a unilateral deviation.  The main contribution
of the paper is the design of provably-correct distributed
continuous-time dynamics that find each of these classes of outcomes.
The problem of finding stable outcomes is combinatorial in the number
of edges in the network. Nevertheless, we build on the correspondence
between the existence of stable outcomes and the solutions of a linear
programming relaxation for the maximum weight matching problem on the
graph to synthesize a distributed algorithm that determines stable
outcomes.  Regarding balanced outcomes, we note how finding them
requires each pair of matched agents to solve a system of coupled
nonlinear equations. Based on this observation, we define local (with
respect to $2$-hop information) error functions that measure how far
matched agents' allocations are from being balanced. Our proposed
distributed algorithm has then agents adjust their allocations based
on their balancing errors.  Finally, we interconnect the two
aforementioned dynamics to synthesize a distributed algorithm that
finds Nash bargaining solutions. We show how the `stable outcome' part
of the dynamics allows agents to guess in a distributed way with whom
to match and that this prediction becomes correct in finite time.
Based on this prediction, the `balanced outcome' part of the dynamics
asymptotically converges to a Nash outcome.  As a byproduct of the
systems and control perspective adopted here, we also assess the
robustness properties of the proposed distributed algorithm against
small perturbations. Our technical approach combines notions and tools
from distributed linear programming, graph theory, nonsmooth analysis,
and Lyapunov stability techniques.  We conclude by applying our
results to a wireless communication scenario in which multiple devices
send data to a base station according to a time division multiple
access protocol. Devices may share their transmission time slots in
order to gain an improved channel capacity. Simulation results show
how agents find a Nash outcome in a distributed way, yielding fair
capacity improvements for each matched device and a network-wide
capacity improvement of around $16\%$.


\emph{Organization.} Section~\ref{sec:prelim} introduces notation and
background material used throughout the
paper. Section~\ref{sec:bargain} introduces the problem of designing
distributed algorithms to find stable, balanced, and Nash outcomes.
Sections~\ref{sec:stable},~\ref{sec:balance}, and~\ref{sec:nash}
provide, respectively, our algorithmic solutions for each of these
classes of outcomes. Section~\ref{sec:sims} presents simulations
results and finally Section~\ref{sec:conclusions} gathers our
conclusions and ideas for future work.

\myclearpage
\section{Preliminaries}\label{sec:prelim}

This section introduces basic preliminaries on notation, nonsmooth
analysis, set-valued dynamical systems, and distributed linear
programming.

\subsection{Notation}

We denote the set of real and nonnegative real numbers by $\reals$ and
$\realsnonnegative$, respectively. For a set $X \subseteq \reals^n$,
its intersection with the nonnegative orthant is denoted $X_+ := X
\cap \realsnonnegative^n$. This notation is applied analogously to
vectors and scalars. For $x \in \reals^n$, we use $x \ge 0$ (resp. $x
> 0$) to mean that all components of $x$ are nonnegative
(resp. positive).
We let $\bar{S}$ denote the closure of the set $S$.  Given sets $S_1$
and $S_2$, $S_1 \setminus S_2$ denotes the complement of $S_2$ in
$S_1$. A set $S \subseteq \real^n$ is convex if it fully contains the
segment connecting any two points in $S$.
The set $\ball(x,\delta) \subseteq \reals^n$ is the open ball centered
at $x \in \reals^n$ with radius $\delta > 0$. We use the shorthand
notation $S + \ball(0,\delta)$ to denote the union $\cup_{x \in
  S}\ball(x,\delta)$.  Given a matrix $A \in \real^{n \times m}$, we
let $A_i$ denote the $i^{\operatorname{th}}$ row of $A$ and $a_{i,j}$
its $(i,j)$ element.

\subsection{Nonsmooth analysis}

Here we review some basic notions from nonsmooth analysis
following~\cite{FHC:83}. A function $f:\reals^n \rightarrow \reals$ is
locally Lipschitz at $x \in \reals^n$ if there exist $\delta_x > 0$
and $L_x \ge 0$ such that $\|f(y_1) - f(y_2)\| \le L_x \|y_1-y_2\|$
for $y_1,y_2 \in \ball(x,\delta_x)$.  We refer to $f$ simply as
locally Lipschitz if $f$ is locally Lipschitz at all $x \in \reals^n$.
A locally Lipschitz function is differentiable almost
everywhere. Letting $\Omega_f \subseteq \reals^n$ be the set of points
where the locally Lipschitz function $f$ is not differentiable, the
generalized gradient of $f$ at $x \in \reals^n$ is
\begin{align*}
  \partial f(x) = \convexhull \Big\{ \lim_{i \rightarrow \infty}
  \nabla f(x_i) : x_i \rightarrow x, x_i \notin S \cup \Omega_f
  \Big\},
\end{align*}
where $\convexhull \{ \cdot \}$ denotes the convex hull and $S
\subseteq \reals^n$ is any set with zero Lebesgue measure. We use
$\partial_xg(x,y)$ and $\partial_xg(x,y)$ to denote the generalized
gradient of the maps $x \mapsto g(x,y)$ and $y \mapsto g(x,y)$,
respectively.

A set-valued map $F:\reals^n \rightrightarrows \reals^n$ maps elements
in $\reals^n$ to subsets of $\reals^n$.  A set-valued map $F$ is
locally bounded if for every $x \in \reals^n$ there exists an
$\epsilon > 0$ such that $F(\ball(x,\epsilon))$ is bounded.  A
set-valued map $F$ is upper semi-continuous if for all $x \in
\reals^n$ and $\epsilon \ge 0$, there exists $\delta_x \ge 0$ such
that $F(y) \subseteq F(x) + \ball(0,\epsilon)$ for all $y \in
\ball(x,\delta_x)$. Given a locally Lipschitz function $f:\reals^n
\rightarrow \reals$, the generalized gradient map $x \mapsto \partial
f(x)$ is a locally bounded and upper semi-continuous set-valued map.
Moreover, $\partial f(x)$ is nonempty, convex, and compact for all $x
\in \reals^n$.


\subsection{Set-valued dynamical systems}

We present here basic notions on set-valued dynamical systems
following the exposition of~\cite{JC:08-csm-yo}. A time-invariant
set-valued dynamical system is represented by the differential
inclusion
\begin{align}
  \dot{x} \in F(x), \label{eq:set_val_dyn}
\end{align}
where $F : X \subseteq \reals^n \rightrightarrows \reals^n$ is a
set-valued map.  If $F$ is locally bounded, upper semi-continuous and
takes nonempty, convex, and compact values, then, from any initial
condition, there exists an absolutely continuous curve $x:\reals_{\ge
  0} \rightarrow X$ (called trajectory or solution)
satisfying~\eqref{eq:set_val_dyn} almost everywhere.  The set of
equilibria of $F$ is $\{ x \in X : 0 \in F(x) \}$. A set $\mathcal{M}$
is weakly positively invariant with respect to $F$ if, for any $x_0
\in \mathcal{M}$, $\mathcal{M}$ contains at least one maximal solution
(that is, one that cannot be extended forward in time) of $\dot{x} \in
F(x)$ with initial condition $x_0$.  The set-valued Lie derivative of
a locally Lipschitz function $V:\reals^n \rightarrow \reals$ along $F$
at $x \in \reals^n$ is defined as
\begin{align*}
  \mathcal{L}_FV(x) = \{ a \in \reals : \exists v \in F(x) \text{
    s.t. } v^T\zeta = a \; \forall \zeta \in \partial V(x) \}.
\end{align*}
One can show that, if $\mathcal{L}_FV(x) \subseteq (-\infty,0]$ for
all $x \in \reals^n$, then $V$ is non-increasing along the
trajectories of~\eqref{eq:set_val_dyn}.

The following result helps establish the asymptotic convergence
properties of~\eqref{eq:set_val_dyn}.

\begin{theorem}\longthmtitle{Set-valued LaSalle Invariance
    Principle}\label{th:lasalle} Assume $V:\reals^n \rightarrow
  \reals$ is differentiable, the trajectories
  of~\eqref{eq:set_val_dyn} are bounded, and $F$ is locally bounded,
  upper semi-continuous and takes nonempty, convex, and compact
  values. If $\mathcal{L}_FV(x) \subseteq (-\infty,0]$ for all $x \in
    X$, then any trajectory $t \mapsto x(t)$ of~\eqref{eq:set_val_dyn}
    starting in $X$ converges to the largest weakly positively
    invariant set $\mathcal{M}$ contained in $\overline{\{ x \in X : 0
      \in \mathcal{L}_FV(x)\}}$.
\end{theorem}

Set-valued dynamical systems are a helpful tool in understanding the
solutions of a differential equation
\begin{align}\label{eq:diff_eq}
  \dot{x} = f(x), 
\end{align}
when $f:\reals^n \rightarrow \reals^n$ is discontinuous. Formally, let
$S_f$ denote the set of points where $f$ is discontinuous. Define the
Filippov set-valued map $\mathcal{F}[f] : \reals^n \rightrightarrows
\reals^n$ by
\begin{align}\label{eq:filippov-set-valued}
  \mathcal{F}[f](x) := \overline{\convexhull}\big\{\lim_{i \rightarrow
    \infty} f(x_i) : x_i \rightarrow x, \; x_i \notin S_f\big\},
\end{align}
where $\overline{\convexhull} \{ \cdot \}$ denotes the closed convex
hull.  If $f$ is measurable and locally bounded, then $\mathcal{F}[f]$
is locally bounded, upper semi-continuous and takes nonempty, convex,
and compact values.  In such case, a solution $t \mapsto x(t)$
of~\eqref{eq:diff_eq} in the Filippov sense is a solution to $ \dot{x}
\in \mathcal{F}[f](x)$.

\subsection{Distributed linear programming}\label{sec:DLP}

Our review here of distributed linear programming follows closely the
exposition in~\cite{SB-LV:04,DR-JC:13-tac}. A standard form linear
program is given by
\begin{align}\label{eq:standard-lp}
  \min\{c^Tx : Ax=b, \; x \ge 0 \},
\end{align}
where $c \in \reals^n, A^{n \times m}$, and $b \in \reals^m$. Its dual
is
\begin{align}\label{eq:dual-lp}
  \max\{-b^Tz : A^Tz \ge c\}.
\end{align}
The point $(x^*,z^*) \in \reals^n \times \reals^m$ satisfies the KKT
conditions for \eqref{eq:standard-lp} if
\begin{align*}
  Ax^* = b, \quad x^* \ge 0, \quad A^Tz^* \ge c, \; \text{ and } \;
  (A^Tz^*+c)^Tx^* = 0.
\end{align*}
When~\eqref{eq:standard-lp} has a finite optimal value, $x^*$
(resp. $z^*$) is a solution to~\eqref{eq:standard-lp} (resp. the
dual~\eqref{eq:dual-lp}) if and only if $(x^*,z^*)$ satisfies the KKT
conditions for~\eqref{eq:standard-lp}.

We have proposed in~\cite{DR-JC:13-tac} the following continuous-time
dynamics to solve the linear program~\eqref{eq:standard-lp},
\begin{subequations}\label{eq:disc_dyn}
  \begin{align}\label{eq:disc_dyn_x}
    \dot{x}_i & = \begin{cases} \nomflow_i(x,z), \hspace{17mm}
      \text{if $x_i>0$,} \\ \max\{0,\nomflow_i(x,z)\}, \quad \text{if
        $x_i = 0$,}
    \end{cases}
    \\
    \dot{z} &= Ax-b, \label{eq:disc_dyn_z}
  \end{align}
\end{subequations}
where the nominal flow function, $\nomflow : \real^n \times \reals^m
\rightarrow \reals^n$, is defined by
\begin{align}\label{eq:nomflow}
  \nomflow = -c-A^T(Ax-b+z).
\end{align}
The convergence properties of this dynamics are captured in the
following result.

\begin{theorem}\longthmtitle{Convergence to a solution of a linear
    program}\label{th:DLP}
  Let $t \rightarrow (x(t),z(t))$ be a trajectory
  of~\eqref{eq:disc_dyn} starting from an initial point in
  $\reals^n_{\ge 0} \times \reals^m$. Then, if~\eqref{eq:standard-lp}
  has a finite optimal value, the following limit exists
  \begin{align*}
    \lim_{t \rightarrow \infty} (x(t),z(t)) = (x^*,z^*),
  \end{align*}
  where $x^*$ (resp. $z^*$) is a solution to~\eqref{eq:standard-lp}
  (resp. the dual of~\eqref{eq:standard-lp}).
\end{theorem}

In addition to the asymptotic convergence of~\eqref{eq:nomflow} stated
in Theorem~\ref{th:DLP}, here we mention two additional important
properties of this dynamics: it is robust to disturbances
(specifically, integral input-to-state stable) and amenable to
distributed implementation. We elaborate on the latter next, as it is
relevant for the main developments of the paper.  Suppose that each
component of $x \in \reals^n$ corresponds to the state of an
independent decision maker or agent.
In order for agent $i \in \until{n}$ to be able to implement its
corresponding dynamics in~\eqref{eq:disc_dyn_x}, it needs access to
the following data and states:
\begin{enumerate}
  \item $c_i \in \reals$,
  \item every $b_{\ell} \in \reals$ for which $a_{\ell,i} \not= 0$,
  \item the non-zero elements of every row of $A$ for which the
    $i^{\text{th}}$ component, $a_{\ell,i}$, is non-zero,
  \item the states of every agent $j \in \until{n}$ where $a_{\ell,i}
    \not= 0$ and $a_{\ell,j} \not=0$ for some $\ell \in
    \{1,\dots,m\}$, and
  \item every $z_{\ell}$ for which $a_{\ell,i} \not= 0$.
\end{enumerate}
The agents $j$ for which $a_{\ell,i} \not= 0$ and $a_{\ell,j} \not=0$
for some $\ell \in \{1,\dots,m\}$ (as in (iv) above) are called
neighbors of $i$, denoted $\neigh(i)$. Also, the states $z_{\ell}$ are
auxiliary states whose dynamics can, based on locally available data
and states, be implemented by any agent $i$ for which $a_{\ell,i}
\not= 0$.

\myclearpage
\section{Problem statement}\label{sec:bargain}

The main objective of the paper is the design of provably correct
distributed dynamics that solve the network bargaining game.  This
section provides a formal description of the problem. We begin by
presenting the model for the group of agents and then recall various
important notions of outcome for the network bargaining game.  

Let $\graph = (\vertices,\edges,\weights)$ be an undirected weighted
graph where $\vertices = \{1, \dots, n\}$ is a set of vertices,
$\edges \subseteq \vertices \times \vertices$ is a set of edges, and
$\weights \in \realsnonnegative^{|\edges|}$ is a vector of edge weights
indexed by edges in $\graph$. In an exchange network, vertices
correspond to agents (or players) and edges connect agents who have
the ability to negotiate with each other. The set of agents that $i$
can negotiate with are its neighbors and is denoted by $\neigh(i) :=
\{j:(i,j) \in \edges\}$. Edge weights represent a transferable utility
that agents may, should they come to an agreement, divide between
them. Here, we assume that the network is a \emph{dyadic-exchange}
network, meaning that agents can pair with at most one other
agent. Agents are selfish and seek to maximize the amount they
receive. However if two agents $i$ and $j$ cannot come to an
agreement, they forfeit the entire amount $\weight_{i,j}$.  We
consider bargaining outcomes of the following form.

\begin{definition}\longthmtitle{Outcomes}\label{def:outcome}
  {\rm A \emph{matching} $M \subseteq \edges$ is a subset of edges
    without common vertices.  An \emph{outcome} is a pair
    $(\matching,\allocation)$, where $\matching \subseteq \edges$ is a
    matching and $\allocation \in \reals^n$ is an allocation to each
    agent such that $\allocation_i + \allocation_j = \weight_{i,j}$ if
    $(i,j) \in \matching$ and $\allocation_k = 0$ if agent $k$ is not
    part of any edge in $\matching$.} \oprocend
\end{definition}

In any given outcome $(\matching,\allocation)$, an agent may decide to
unilaterally deviate by matching with another neighbor. As an example,
suppose that $(i,j) \in \matching$ and agent $k$ is a neighbor of $i$.
If $\allocation_i + \allocation_k < w_{i,k}$, then there is an
incentive for $i$ to deviate because it could receive an increased
allocation of $\hat{\allocation}_i = w_{i,k} - \allocation_k >
\allocation_i$. Such a deviation is unilateral because $k$'s
allocation stays constant. Conversely, if $\allocation_i +
\allocation_k \ge w_{i,k}$, then $i$ does not have an incentive to
deviate by matching with~$k$. This discussion motivates the notion of
a \emph{stable outcome}, in which no agent benefits from a unilateral
deviation.

\begin{definition}\longthmtitle{Stable outcome}
  {\rm An outcome $(\matching, \allocation^s)$ is stable if
    $\allocation^s \ge 0$ and
    \begin{align*}
      \hspace{-7mm} \allocation^s_i + \allocation^s_j \ge
      \weight_{i,j}, \quad \forall (i,j) \in \edges. \eqoprocend
    \end{align*}} 
\end{definition}

Given an arbitrary matching $\matching$, it is not always possible to
find allocations $\allocation^s$ such that $(\matching,\allocation^s)$
is a stable outcome. Thus, finding stable outcomes requires to find an
appropriate matching as well, making the problem combinatorial in the
number of possible matchings.

Stable outcomes are not necessarily fair between matched agents, and
this motivates the notion of \emph{balanced} outcomes. As an example,
again assume that the outcome $(\matching,\allocation^b)$ is given and
that $(i,j) \in \matching$. The \emph{best allocation} that $i$ could
expect to receive by matching with a neighbor other than $j$ is
\begin{align*}
  \besta{i}{j}(\allocation^b) = \max_{k \in \neigh(i) \setminus j} \{
  \weight_{i,k} - \allocation^b_k \}_+.
\end{align*} 
Moreover, the set (possibly empty) of \emph{best neighbors} with whom
$i$ could receive this allocation is
\begin{align*}
  \nextbest{i}{j}(\allocation^b) = \argmax_{k \in \neigh(i) \setminus
    j} \{ \weight_{i,k} - \allocation^b_k \}_+.
\end{align*}
Then, if agent $i$ were to unilaterally deviate from the outcome and
match instead with $k \in \nextbest{i}{j}$, the resulting benefit of
this deviation would be
\begin{align*}
  \besta{i}{j}(\allocation^b) - \allocation^b_i.
\end{align*}
When the benefit of a deviation is the same for both $i$ and $j$, we
call the outcome balanced. 

\begin{definition}\longthmtitle{Balanced outcome}\label{def:balanced}
  {\rm An outcome $(\matching,\allocation^b)$ is balanced if for
    all $(i,j) \in \matching$,
    \begin{align*}
      \besta{i}{j}(\allocation^b) - \allocation^b_i =
      \besta{j}{i}(\allocation^b) - \allocation^b_j. \eqoprocend
    \end{align*}}
\end{definition}

From its definition, it is easy to see that the main challenge in
finding balanced outcomes is the fact that the allocations must
satisfy a system of nonlinear (in fact, piecewise linear) equations,
coupled between agents. Of course, outcomes that are both stable and
balanced are desirable and what we seek in this paper. Such outcomes
are called \emph{Nash}.

\begin{definition}\longthmtitle{Nash outcome}
   {\rm An outcome $(\matching,\allocation^N)$ is Nash if it is stable
     and balanced.} \oprocend
\end{definition}

Figure~\ref{fig:outcomes} shows an example of each outcome,
highlighting their various attributes.
\begin{figure}[hbt!]
  \centering
  \subfigure[Stable]{\includegraphics[width=.32\linewidth]{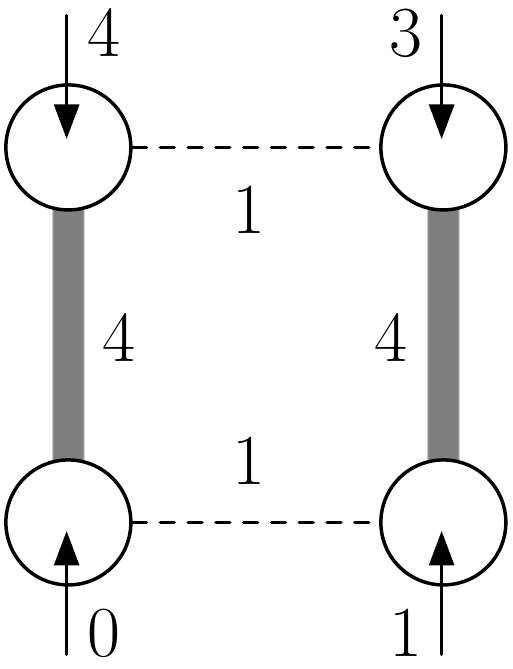}}
  \subfigure[Balanced]{\includegraphics[width=.32\linewidth]{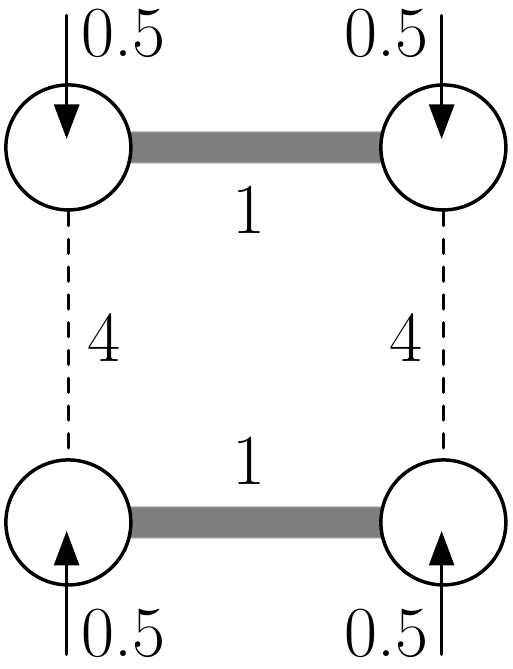}}
  \subfigure[Nash]{\includegraphics[width=.32\linewidth]{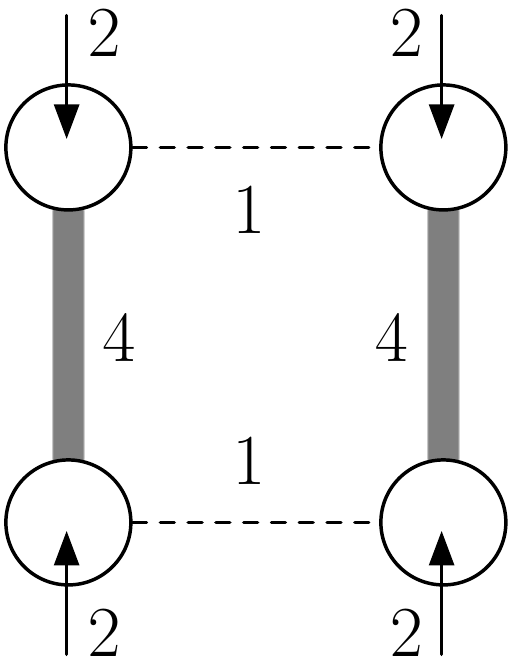}}
  \caption{For each outcome, matched agents are connected with thicker
    grey edges, dotted edges connect agents who decided not to match,
    and allocations are indicated by arrows. In the stable
    outcome, the $0$ allocation is unfair to that node since its
    partner receives the whole edge weight. In the balanced outcome,
    agents can receive higher allocations by deviating from their
    matches. Nash outcomes do not exhibit either of these
    shortcomings.}\label{fig:outcomes}
\end{figure}

The problem we aim to solve is to develop distributed dynamics that
converge to each of the class of outcomes defined above: stable,
balanced, and Nash.  We refer to a dynamics as \emph{$1$-hop
  distributed}, or simply \emph{distributed, over $\graph$} if its
implementation requires each agent $i \in \until{n}$ only knowing (i)
the states of $1$-hop neighboring agents and (ii) the utilities
$w_{i,j}$ for each $j \in \neigh(i)$. Likewise, we refer to a dynamics
as \emph{$2$-hop distributed over $\graph$} if its implementation
requires each agent $i \in \until{n}$ only knowing (i) the states of
$1$- and $2$-hop neighboring agents and (ii) the utilities $w_{i,j}$
and $w_{j,k}$ for each $j \in \neigh(i)$ and $k \in \neigh(j)$.  As
agents' allocations evolve in the dynamics that follow, the quantity
$\weight_{i,j} - \allocation_i(t)$ has the interpretation of ``$i$'s
\emph{offer} to $j$ at time $t$'', thus motivating the terminology
\emph{bargaining} in exchange networks.

\myclearpage
\section{Distributed dynamics to find stable
  outcomes}\label{sec:stable}

In this section, we propose a distributed dynamics to find stable
outcomes in network bargaining. Our strategy to achieve this builds on
a reformulation of the problem of finding a stable outcome in terms of
finding the solutions to a linear program.

\subsection{Stable outcomes as solutions of linear program}

Here we relate the existence of stable outcomes to the solutions of a
linear programming relaxation for the maximum weight matching
problem. This reformulation allows us later to synthesize a
distributed dynamics to find stable outcomes. Our discussion here
follows~\cite{JK-ET:08}, but because that reference does not present
formal proofs of the results we need, we include them here for
completeness.

We begin by recalling the formulation of the \emph{maximum weight
  matching problem} on~$\graph$.  Essentially, this corresponds to a
matching in which the sum of the edge weights in the matching is
maximal. Formally, for every $(i,j) \in \edges$ we use variables
$m_{i,j} \in \{0,1\}$ to indicate whether $(i,j)$ is in the maximum
weight matching (i.e., $m_{i,j} = 1$) or not ($m_{i,j} = 0$). Then,
the solutions of the following integer program can be used to deduce a
maximum weight matching,
\begin{subequations}\label{eq:mwm}
  \begin{alignat}{2}
    &\max && \quad \text{\footnotesize $\sum$}_{(i,j) \in \edges}
    \weight_{i,j} m_{i,j} \\ &\hspace{1.5mm}\text{s.t.} && \quad
    \text{\footnotesize $\sum$}_{j \in \neigh(i)} m_{i,j} \le 1, \quad
    \forall i \in \vertices, \label{eq:con_match}\\ & && \quad m_{i,j}
    \in \{0, 1\}, \quad \forall (i,j) \in
    \edges. \label{eq:combinatorial}
  \end{alignat}
\end{subequations}
The constraints~\eqref{eq:con_match} ensure that each agent is matched
to at most one other agent. If $m^* \in \{0,1\}^{|\edges|}$ is a
solution (indexed by edges in $\graph$) to the above optimization
problem, then a maximum weight matching is well-defined by the
relationship $(i,j) \in \matching \Leftrightarrow m^*_{i,j} =
1$. Since~\eqref{eq:mwm} is combinatorial in the number of edges in
the graph due to constraint~\eqref{eq:combinatorial}, we are
interested in studying its linear programming relaxation,
\begin{alignat}{2}\label{eq:max_match}
  &\max && \quad \text{\footnotesize $\sum$}_{(i,j) \in \edges}
  \weight_{i,j} m_{i,j} \nonumber
  \\
  &\hspace{1.5mm}\text{s.t.} && \quad \text{\footnotesize $\sum$}_{j
    \in \neigh(i)} m_{i,j} \le 1, \quad \forall i \in \vertices,
  \\
  & && \quad m_{i,j} \ge 0, \quad \forall (i,j) \in \edges \nonumber,
\end{alignat}
and its associated dual
\begin{alignat}{2}\label{eq:match_dual}
  &\min && \quad \text{\footnotesize $\sum$}_{i \in \vertices}
  \allocation^s_i \nonumber
  \\
  &\hspace{1.5mm}\text{s.t.} && \quad \allocation^s_i +
  \allocation^s_j \ge \weight_{i,j}, \quad \forall (i,j) \in \edges,
  \\
  & && \quad \allocation^s_{i} \ge 0, \quad \forall i \in
  \vertices. \nonumber
\end{alignat}
Arguing with the KKT conditions for the
relaxation~\eqref{eq:max_match}, the following result states that when
a stable outcome $(\matching,\allocation^s)$ exists, the matching
$\matching$ is a maximum weight matching on $\graph$.

\begin{lemma}\longthmtitle{Maximum weight matchings and stable
    outcomes~\cite{JK-ET:08}}\label{lem:mwm}
  Suppose that a stable outcome $(M,\allocation^s)$ exists for a given
  graph $\graph$. Then $M$ is a maximum weight matching.
\end{lemma}
\begin{proof}
  Our proof method is to encode the matching $M$ using the indicator
  variables $m \in \{0,1\}^{|\edges|}$ and then show that $m$ is a
  solution of the maximum weight matching problem~\eqref{eq:mwm}. To
  begin, for all $(i,j) \in \edges$, let $m_{i,j} = 1$ if $(i,j) \in
  \matching$ and zero otherwise. Use $m$ to denote the vector of
  $m_{i,j}$, indexed by edges in $\graph$. Then $m$ is feasible for
  the relaxation~\eqref{eq:max_match}. By definition of outcome,
  cf. Definition~\ref{def:outcome}, it holds that
  $m_{i,j}(\allocation^s_i + \allocation^s_j - \weight_{i,j}) = 0$ for
  all $(i,j) \in \edges$ and $\allocation^s_i(1-\sum_{j \in \neigh(i)}
  m_{i,j}) = 0$ for all $i \in \vertices$. In other words,
  complementary slackness is satisfied. Also, note that
  $\allocation^s$ is feasible for the dual~\eqref{eq:match_dual}. This
  means that $(m,\allocation^s)$ satisfy the KKT conditions
  for~\eqref{eq:max_match} and so $m$ is a solution
  of~\eqref{eq:max_match}. Since $m$ is integral, it is also a
  solution of~\eqref{eq:mwm} implying that $M$ is a maximum weight
  matching. This completes the proof.
\end{proof}

Building on this result, we show next that the existence of stable
outcomes is directly related to the existence of integral solutions of
the linear programming relaxation~\eqref{eq:max_match}.

\begin{lemma}\longthmtitle{Existence of stable
    outcomes~\cite{JK-ET:08}}\label{lem:exist}
  A stable outcome exists for the graph $\graph$ if and only
  if~\eqref{eq:max_match} admits an integral solution. Moreover, if
  $\graph$ admits a stable outcome and $m^* \in \{0,1\}^{|\edges|}$ is
  a solution to~\eqref{eq:max_match}, then
  $(\matching,\allocation^{s,*})$ is a stable outcome where the
  matching $\matching$ is well-defined by the implication
  \begin{align}
    (i,j) \in \matching \Leftrightarrow m^*_{i,j} =
    1, \label{eq:induced_matching}
  \end{align}
  and $\allocation^{s,*}$ is a solution to~\eqref{eq:match_dual}.
\end{lemma}
\begin{proof}
  The proof of Lemma~\ref{lem:mwm} revealed that, if a stable outcome
  exists,~\eqref{eq:max_match} admits an integral solution. Let us
  prove the other direction. By assumption,~\eqref{eq:max_match}
  yields an integral solution $m^* \in \{0,1\}^{|\edges|}$ and let
  $\allocation^{s,*}$ be a solution to the
  dual~\eqref{eq:match_dual}. We induce the following matching: $(i,j)
  \in \matching$ if $m^*_{i,j} = 1$ and $(i,j) \notin \matching$
  otherwise. By complementary slackness,
  $m^*_{i,j}(\allocation^{s,*}_i + \allocation^{s,*}_j -
  \weight_{i,j}) = 0$ for all $(i,j) \in \edges$ and
  $\allocation^{s,*}_i(1-\sum_{j \in \neigh(i)} m^*_{i,j}) = 0$ for
  all $i \in \vertices$. Then, it must be that $m^*_{i,j} = 1$ implies
  that $\allocation^{s,*}_i + \allocation^{s,*}_j = w_{i,j}$ and
  $\allocation^{s,*}_i = 0$ if $i$ is not part of any matching. Thus,
  $(M,\allocation^{s,*})$ is a valid outcome. Next,
  $\allocation^{s,*}$ must be feasible for~\eqref{eq:match_dual},
  which reveals that it is a stable allocation. Therefore,
  $(M,\allocation^{s,*})$ is a stable outcome. This completes the
  proof.
\end{proof}


\subsection{Stable outcomes via distributed linear programming}

Since we are interested in finding stable outcomes, from here on we
make the standing assumption that one exists and that the maximum
weight matching is unique. Besides its technical implications,
requiring uniqueness has a practical motivation and is a standard
assumption in exchange network bargaining. For example, if an agent
has two equally good alternatives, it is unclear with whom it will
choose to match with. It turns out that the set of graphs for which a
unique maximum weight matching exists is open and dense in the set of
graphs that admit a stable outcome, further justifying the assumption
of uniqueness of the maximum weight matching.

Given the result in Lemma~\ref{lem:exist} above, finding a stable
outcome is a matter of solving the relaxed maximum weight matching
problem, where the matching is induced from the solution
of~\eqref{eq:max_match} and the allocation is a solution
to~\eqref{eq:match_dual}. Our next step is to
put~\eqref{eq:match_dual} in standard form by introducing slack
variables $\slack_{i,j}$ for each $(i,j) \in \edges$,
\begin{subequations}\label{eq:stable_slack}
  \begin{alignat*}{2}
    &\min && \quad \text{\footnotesize $\sum$}_{i \in \vertices}
    \allocation^s_i \\ &\hspace{1.5mm}\text{s.t.} && \quad
    \allocation^s_i + \allocation^s_j - \slack_{i,j} = \weight_{i,j},
    \quad \forall (i,j) \in \edges, \\ & && \quad \allocation^s_{i}
    \ge 0 \quad \forall i \in \vertices, \\ & && \quad
    \slack_{i,j} \ge 0 \quad \forall (i,j) \in \edges.
  \end{alignat*}
\end{subequations}
We use the notation $s$ to represent the vector of slacks indexed by
edges in $\graph$. In the dynamics that follow, the variables $s$ and
$m$ will be states. Thus, as a convention, we assume that each
$s_{i,j}$ and $m_{i,j}$ are states of agent $\min\{i,j\}$. This means
that the state of agent $i \in \vertices$ is
\begin{align*}
  (\allocation^s_i,\{\slack_{i,j}\}_{j \in
    \neigh^+(i)},\{m_{i,j}&\}_{j \in \neigh^+(i)}) \\ &\in \reals_{\ge
    0} \times \realsnonnegative^{|\neigh^+(i)|} \times
  \reals^{|\neigh^+(i)|},
\end{align*}
where, for convenience, we denote by $\neigh^+(i) := \{j \in
\neigh(i): i < j\}$ the set of neighbors of $i$ whose identity is
greater than~$i$.

Next, using the dynamics~\eqref{eq:disc_dyn} of Section~\ref{sec:DLP}
to solve the linear program above results in the following dynamics
for agent $i \in \until{n}$,
\begin{subequations}\label{eq:stable_dyn}
  \begin{align}
    \dot{\allocation}^s_i = \left\{ 
      \begin{alignedat}{2}
        &\nomflow^{\allocation}_i(\allocation^s,\slack,m), && \quad
        \allocation^s_i > 0,
        \\
        & \max\{0,\nomflow^{\allocation}_i(\allocation^s,\slack,m)\},
        && \quad \allocation^s_i = 0,
      \end{alignedat}\right.
  \end{align}
  and, for each $j \in \neigh^+(i)$
  \begin{align}
    \dot{\slack}_{i,j} &= \left\{
      \begin{alignedat}{2}
        &\nomflow^{\slack}_{i,j}(\allocation^s,\slack,m), && \quad
        \slack_{i,j} > 0,
        \\
        &\max\{0,\nomflow^{\slack}_{i,j}(\allocation^s,\slack,m)\}, &&
        \quad \slack_{i,j} = 0,
      \end{alignedat} \right.
    \\
    \dot{m}_{i,j} &= \allocation^s_i + \allocation^s_j - \slack_{i,j}
    - \weight_{i,j},
  \end{align}
\end{subequations}
where
\begin{align*}
  \nomflow^{\allocation}_i(\allocation^s,\slack,m) := - 1 \! - \!
  \sum_{j \in \neigh(i)} \bigg[m_{i,j} + \allocation^s_i +
    \allocation^s_j - \slack_{i,j} - \weight_{i,j}\bigg],
\end{align*}
and
\begin{align*} 
  \nomflow^{\slack}_{i,j}(\allocation^s,\slack,m) := - m_{i,j} -
  \allocation^s_i - \allocation^s_j + \slack_{i,j} + \weight_{i,j},
\end{align*}
are derived from~\eqref{eq:nomflow}. The next result reveals how this
dynamics can be used as a distributed algorithm to find stable
outcomes.

\begin{proposition}\longthmtitle{Convergence to stable
    outcomes}\label{prop:stab}
  Given a graph $\graph$, let $t \rightarrow
  (\allocation^s(t),\slack(t),m(t))$ be a trajectory
  of~\eqref{eq:stable_dyn} starting from an initial point in
  $\reals^n_{\ge 0} \times \realsnonnegative^{|\edges|} \times
  \reals^{|\edges|}$. Then the following limit exists
  \begin{align*}
    \lim_{t \rightarrow \infty} (\allocation^s(t),\slack(t),m(t)) =
    (\allocation^*,\slack^*,m^*),
  \end{align*}
  where $(\allocation^{s,*},\slack^*)$ (resp. $m^*$) is a solution
  to~\eqref{eq:match_dual} (resp.~\eqref{eq:max_match}). Moreover, if
  a stable outcome exists, a maximum weight matching $M$ is
  well-defined by the implication
  \begin{align*}
    (i,j) \in \matching \quad \Leftrightarrow \quad m^*_{i,j} = 1,
  \end{align*}
  and $(M,\allocation^{s,*})$ is a stable outcome. Finally, the
  dynamics~\eqref{eq:stable_dyn} is distributed over $\graph$.
\end{proposition}

The proof of the above results follows directly from
Theorem~\ref{th:DLP}, Lemma~\ref{lem:exist}, and the assumptions made
on the information available to each agent.

\myclearpage
\section{Distributed dynamics to find balanced
  outcomes}\label{sec:balance}

In this section, we introduce distributed dynamics that converge to
balanced outcomes. Our starting point is the availability of a
matching~$\matching$ to the network, i.e., each agent already knows if
it is matched and who its partner is. Hence, the dynamics focuses on
negotiating the allocations to find a balanced one. We drop this
assumption later when considering Nash outcomes.

Our algorithm design is based on the observation that the condition
$\allocation^b_i + \allocation^b_j = \weight_{i,j}$ for $(i,j) \in
\matching$ that defines an allowable allocation for an outcome and the
balance condition in Definition~\ref{def:balanced} for two matched
agents can be equivalently stated as
\begin{subequations}\label{eq:balance_rewritten}
  \begin{align}
    0 \! &= \! \allocation^b_i \! - \! \frac{1}{2}(\weight_{i,j} \! +
    \besta{i}{j}(\allocation^b) \! - \besta{j}{i}(\allocation^b)) =:
    e^b_i(\allocation^b), \label{eq:balance_rewritten_i}
    \\
    0 \! &= \!  \allocation^b_j \! - \! \frac{1}{2}(\weight_{i,j} \! -
    \besta{i}{j}(\allocation^b) \! + \besta{j}{i}(\allocation^b)) =:
    e^b_j(\allocation^b). \label{eq:balance_rewritten_j}
  \end{align}
\end{subequations}
We refer to $e^b_i,e^b_j:\reals^n \rightarrow \reals$ as the errors
with respect to satisfying the balance condition of $i$ and $j$,
respectively. For an unmatched agent $k$, we define $e^b_{k} =
\allocation^b_k$. We refer to $e^b(\allocation^b) \in \reals^n$ as the
vector of \emph{balancing errors} for a given allocation. Based on the
observation above, we propose the following distributed dynamics
whereby agents adjust their allocations proportionally to their
balancing errors,
\begin{align} \label{eq:balance_dyn} \dot{\allocation}^b &=
  -e^b(\allocation^b).
\end{align}
An important fact to note is that the equilibria of the above dynamics
are, by construction, allocations in a balanced outcome. Also, note
that~\eqref{eq:balance_dyn} is continuous and requires agents to know
$2$-hop information, because for its pair of matched agents $(i,j) \in
\matching$, agent $i$ updates its own allocation (and hence its offer
to $j$) based on $\besta{j}{i}$. 

The following result establishes the boundedness of the balancing
errors under~\eqref{eq:balance_dyn} and is useful later in
establishing the asymptotic convergence of this dynamics to an
allocation in a balanced outcome with matching~$\matching$.

\begin{figure*}[!b] 
  \hrulefill \normalsize \setcounter{MYtempeqncnt}{\value{equation}}
  \setcounter{equation}{16}
  \begin{align}\label{eq:lie-V}
    &\mathcal{L}_{\mathcal{F}[F]}V(\xi) = \{a \in \reals :
    \text{there exists } v \in \mathcal{F}[F](\xi,\omega) \text{
      s.t. } a = \zeta^Tv \text{ for all } \zeta \in \partial V(\xi)\}
    \nonumber
    \\
    & \; = \bigg \{a \in \reals : \text{for each } i \in \vertices
    \text{ there exists } \lambda^i \in \reals^n_{\ge 0} \text{ with }
    \hspace{-5mm} \sum_{\tau \in \nextbest{i}{j}(\omega)}\hspace{-5mm}
    \lambda^i_{\tau} = 1 \text{ if } \nextbest{i}{j}(\omega) \not=
    \emptyset \text{ and } \mu^i \in \reals^n_{\ge 0} \text{ with }
    \hspace{-5mm} \sum_{\kappa \in \nextbest{j}{i}(\omega)}
    \hspace{-5mm}\mu^i_{\kappa} = 1 \text{ if } \nonumber
    \\
    & \hspace{8mm} \nextbest{j}{i}(\omega) \not= \emptyset \text{
      s.t. } a \!  = \!  \bigg( \sum_{i \in \vertices} h_i \bigg[ \!-
    \!\xi_i - \tfrac{1}{2}\hspace{-5mm}\sum_{\tau \in
      \nextbest{i}{j}(\omega)}\hspace{-5mm} \lambda^i_{\tau}
    \xi_{\tau} + \tfrac{1}{2}\hspace{-5mm}\sum_{\kappa \in
      \nextbest{j}{i}(\omega)}\hspace{-5mm}
    \mu^i_{\kappa}\xi_{\kappa}\bigg]\bigg)^T\!\!\bigg( \sum_{ i \in
      \maxerrors(\xi)}\hspace{-2mm}\eta_i h_i \xi_i\bigg) \text{ for
      all } \eta \in \reals^n_{\ge 0} \text{ with} \hspace{-2mm} \sum_{ i \in
      \maxerrors(\xi)}\hspace{-2mm} \eta_i = 1 \bigg\} \nonumber
    \\
    & \; = \bigg \{a \in \reals : \text{for each } i \in
    \maxerrors(\xi) \text{ there exists } \lambda^i \in \reals^n_{\ge 0}
    \text{ with } \hspace{-5mm} \sum_{\tau \in
      \nextbest{i}{j}(\omega)}\hspace{-5mm} \lambda^i_{\tau} = 1
    \text{ if } \nextbest{i}{j}(\omega) \not= \emptyset \text{ and }
    \mu^i \in \reals^n_{\ge 0} \text{ with } \hspace{-5mm} \sum_{\kappa \in
      \nextbest{j}{i}(\omega)} \hspace{-5mm}\mu^i_{\kappa} = 1 \text{
      if } \nonumber
    \\
    & \hspace{8mm} \nextbest{j}{i}(\omega) \not= \emptyset \text{
      s.t. } a = \hspace{-2mm} \sum_{i \in \maxerrors(\xi)} \eta_i
    \bigg(-\xi_i^2 - \tfrac{1}{2}\hspace{-5mm}\sum_{\tau \in
      \nextbest{i}{j}(\omega)}\hspace{-5mm} \lambda^i_{\tau}
    \xi_i\xi_{\tau} + \tfrac{1}{2}\hspace{-5mm}\sum_{\kappa \in
      \nextbest{j}{i}(\omega)}\hspace{-5mm}
    \mu^i_{\kappa}\xi_i\xi_{\kappa}\bigg) \text{ for all } \eta \in
    \reals^n_{\ge 0} \text{ with} \hspace{-2mm} \sum_{ i \in
      \maxerrors(\xi)}\hspace{-2mm} \eta_i = 1 \bigg\}.
  \end{align}
  \vspace*{4pt} 
\end{figure*}

\begin{proposition}\longthmtitle{Balancing errors are
    bounded}\label{prop:bounded}
  Given a matching $\matching$, let $t \mapsto \allocation^b(t)$ be a
  trajectory of~\eqref{eq:balance_dyn} starting from any point in
  $\reals^n$. Then
  \begin{align*}
    t \mapsto V(e^b(\allocation^b(t))) := \tfrac{1}{2} \max_{i \in
      \vertices} (e^b_i(\allocation^b(t)))^2
  \end{align*}
  is non-increasing. Thus, $t \mapsto e^b(\allocation^b(t))$ lies in a
  bounded set.
\end{proposition}
\begin{proof}
  Our proof strategy is to compute, for each $i \in \vertices$, the
  Lie derivative of $e^b_i$ along the trajectories
  of~\eqref{eq:balance_dyn}. Based on these Lie derivatives, we
  introduce a new dynamics whose trajectories contain $t \mapsto
  e^b(\allocation^b(t))$ and establish the result reasoning with it.

  Since $e^b_i$ is locally Lipschitz, it is differentiable almost
  everywhere. Let $\Omega_i \subseteq \reals^n$ be the set, of measure
  zero, of allocations for which $e^b_i$ is not differentiable. If $i$
  is matched, say $(i,j) \in \matching$, then $\Omega_i$ is precisely
  the set of allocations where at least one of the next best neighbor
  sets $\nextbest{i}{j}(\allocation^b)$ or
  $\nextbest{j}{i}(\allocation^b)$ have more than one element. If $i$
  is unmatched, then $\Omega_i = \emptyset$.  Then, whenever
  $\allocation^b \in \reals^n \setminus \Omega_i$, it is easy to see
  that for every $i \in \vertices$,
  \begin{align*}
    \mathcal{L}_{-e^b}e^b_i(\allocation^b) =
    \begin{cases}
      -e^b_i(\allocation^b), &\text{if $i$ is unmatched or
      }
      \\
      &\text{if }\nextbest{i}{j}(\allocation^b) = \emptyset
      \\
      &\text{and } \nextbest{j}{i}(\allocation^b) = \emptyset
      \\
      -e^b_i(\allocation^b) + \tfrac{1}{2} e^b_{\tau}(\allocation^b),
      &\text{if } \nextbest{i}{j}(\allocation^b) = \tau
      \\
      &\text{and } \nextbest{j}{i}(\allocation^b) = \emptyset
      \\
      -e^b_i(\allocation^b) - \tfrac{1}{2}
      e^b_{\kappa}(\allocation^b), &\text{if }
      \nextbest{i}{j}(\allocation^b) = \emptyset
      \\
      &\text{and } \nextbest{j}{i}(\allocation^b) = \kappa
      \\
      -e^b_i(\allocation^b) + \tfrac{1}{2} e^b_{\tau}(\allocation^b) &
      \hspace{-3.5mm} - \;\tfrac{1}{2} e^b_{\kappa}(\allocation^b),
      \\
      &\text{if } \nextbest{i}{j}(\allocation^b) = \tau
      \\
      & \text{and }\nextbest{j}{i}(\allocation^b) = \kappa.
    \end{cases}
  \end{align*}
  This observation motivates our study of the dynamics
  \setcounter{equation}{15}
  \begin{subequations}\label{eq:xi-omega-dyn}
    \begin{align}
      \dot{\xi}_i =
      \begin{cases}
        -\xi_i, &\text{if $i$ is unmatched or }
        \\
        &\text{if }\nextbest{i}{j}(\omega) = \emptyset
        \\
        &\text{and } \nextbest{j}{i}(\omega) = \emptyset
        \\
        -\xi_i + \tfrac{1}{2} \xi_{\tau}, &\text{if }
        \nextbest{i}{j}(\omega) = \tau
        \\
        &\text{and } \nextbest{j}{i}(\omega) = \emptyset
        \\
        -\xi_i - \tfrac{1}{2} \xi_{\kappa}, &\text{if }
        \nextbest{i}{j}(\omega) = \emptyset
        \\
        &\text{and } \nextbest{j}{i}(\omega) = \kappa
        \\
        -\xi_i + \tfrac{1}{2}\xi_{\tau} - \tfrac{1}{2} \xi_{\kappa},
        &\text{if } \nextbest{i}{j}(\omega) = \tau
        \\
        & \text{and }\nextbest{j}{i}(\omega) = \kappa.
      \end{cases}
      \\
      & \hspace{-70.5mm} \dot{\omega}_i = -\xi_i
    \end{align}
  \end{subequations}
  for every $i \in \vertices$, defined on $\reals^n \times (\reals^n
  \setminus \Omega)$, where $\Omega := \cup_{i \in \vertices}
  \Omega_i$.  For convenience, we use the shorthand notation $F =
  (F^1,F^2) : \reals^n \times (\reals^n \setminus \Omega) \rightarrow
  \reals^n \times \reals^n$ to refer to~\eqref{eq:xi-omega-dyn}. Note
  that $F$ is piecewise continuous (because $F^1$ is piecewise
  continuous, while $F^2$ is continuous). Therefore, we understand its
  trajectories in the sense of
  Filippov. Using~\eqref{eq:filippov-set-valued}, we compute the
  Filippov set-valued map, defined on $ \reals^n \times \reals^n$, for
  any matched $i$ and $(\xi,\omega) \in \reals^n \times \reals^n$, as
  \begin{align*}
    \mathcal{F}[F^1_i](&\xi,\omega) =
    \Big\{  -\xi_i - \tfrac{1}{2}\hspace{-5mm}\sum_{\tau \in
      \nextbest{i}{j}(\omega)}\hspace{-5mm} \lambda^i_{\tau}
    \xi_{\tau} + \tfrac{1}{2}\hspace{-5mm}\sum_{\kappa \in
      \nextbest{j}{i}(\omega)}\hspace{-5mm} \mu^i_{\kappa}\xi_{\kappa}
    :
    \\
    &\quad \lambda^i \in \reals^n_{\ge 0} \text{ is s.t. } \hspace{-5mm}
    \sum_{\tau \in \nextbest{i}{j}(\omega)}\hspace{-5mm}
    \lambda^i_{\tau} = 1 \text{ if } \nextbest{i}{j}(\omega) \not=
    \emptyset
    \\
    & \quad \text{and } \mu^i \in \reals^n_{\ge 0} \text{ is s.t. }
    \hspace{-5mm} \sum_{\kappa \in \nextbest{j}{i}(\omega)}
    \hspace{-5mm}\mu^i_{\kappa} = 1 \text{ if }
    \nextbest{j}{i}(\omega) \not= \emptyset \Big\}.
  \end{align*}
  Here, we make the convention that the empty sum is zero. If $i$ is
  unmatched, then $\mathcal{F}[F^1_i](\xi,\omega) =
  \{-\xi_i\}$. Furthermore, $ \mathcal{F}[F^2_i] = \{-\xi_i\}$ for all
  $i \in \vertices$.  Based on the discussion so far, we know that $t
  \mapsto (e^b(\allocation^b(t)),\allocation^b(t))$ is a Filippov
  trajectory of~\eqref{eq:xi-omega-dyn} with initial condition
  $(e^b(\allocation^b(0)),\allocation^b(0)) \in \reals^n \times
  \reals^n$. Thus, to prove the result, it is sufficient to establish
  the monotonicity of
  \begin{align*}
    V(\xi) = \max_{i \in \vertices} \frac{1}{2} \xi_i^2 ,
  \end{align*}
  along~\eqref{eq:xi-omega-dyn}. For notational purposes, we denote
  \begin{align*}
    \maxerrors(\xi) := \argmax_{i \in \vertices} \frac{1}{2}\xi_i^2.
  \end{align*}
  The generalized gradient of $V$ is
  \begin{align*}
    \partial V(\xi) = \Big\{ \sum_{i \in \maxerrors(\xi)}
    \hspace{-2mm} \eta_i h_i \xi_i : \eta \in \reals^n_{\ge 0} \text{ s.t. }
    \hspace{-2mm} \sum_{i \in \maxerrors(\xi)} \hspace{-2mm} \eta_i =
    1 \Big\},
  \end{align*}
  where $h_i \in \reals^n$ is the unit vector with $1$ in its
  $i^{\operatorname{th}}$ component and $0$ elsewhere. Then, the
  set-valued Lie derivative of $V$ along $\mathcal{F}[F]$ is given
  in~\eqref{eq:lie-V}. To upper bound the element in
  $\mathcal{L}_{\mathcal{F}[F]}V(\xi)$, note that
  \begin{align*}
    -\tfrac{1}{2}\hspace{-5mm}\sum_{\tau \in
      \nextbest{i}{j}(\omega)}\hspace{-5mm} \lambda^i_{\tau}
    \xi_i\xi_{\tau} \le \tfrac{1}{4}\hspace{-5mm} \sum_{\tau \in
      \nextbest{i}{j}(\omega)}\hspace{-5mm}\lambda^i_{\tau}(\xi_i^2 +
    \xi_{\tau}^2),
  \end{align*}
  where we have used the inequality $ab \le \tfrac{1}{2}a^2 +
  \tfrac{1}{2}b^2$ for $a,b \in \reals$. For $\sum_{\tau \in
    \nextbest{i}{j}(\omega)} \lambda^i_{\tau} = 1$ and $i \in
  \maxerrors(\xi)$ (that is $\xi_i^2 \ge \xi_{\tau}^2$ for all $\tau
  \in \nextbest{i}{j}(\omega)$), we can further refine the bound as,
  \begin{align*}
    \tfrac{1}{2}\hspace{-5mm}\sum_{\tau \in
      \nextbest{i}{j}(\omega)}\hspace{-5mm} \lambda^i_{\tau}
    \xi_i\xi_{\tau} \le \tfrac{1}{2} \xi_i^2,
  \end{align*}
  The analogous bound
  \begin{align*}
    \tfrac{1}{2}\hspace{-5mm}\sum_{\kappa \in
      \nextbest{j}{i}(\omega)}\hspace{-5mm} \mu^i_{\kappa}
    \xi_i\xi_{\kappa} \le \tfrac{1}{2} \xi_i^2,
  \end{align*}
  can be derived similarly if $\sum_{\kappa \in
    \nextbest{j}{i}(\omega)} \mu^i_{\kappa} = 1$ and $i \in
  \maxerrors(\xi)$. Using these bounds in the Lie
  derivative~\eqref{eq:lie-V} and noting that $\sum_{i \in
    \maxerrors(\xi)}\eta_i = 1$, it is straightforward to see that for
  any element $a \in \mathcal{L}_{\mathcal{F}[F]}V(\xi)$ it holds
  that $a \le 0$. It follows that $t \mapsto V(\xi(t))$ and thus $t
  \mapsto V(e^b(\allocation^b(t)))$ is non-increasing and $t \mapsto
  e^b(\allocation^b(t))$ lies in the bounded set
  $V^{-1}(e^b(\allocation^b(0)))$, which completes the proof.
\end{proof}

The next result establishes the local stability of the balanced
allocations associated with a given matching and plays a key role
later in establishing the global asymptotic pointwise convergence of
the dynamics~\eqref{eq:balance_dyn}.

\begin{proposition}\longthmtitle{Local stability of each balanced
    allocation}\label{prop:local_stability}
  Given a matching $\matching \subseteq \edges$, let $\BB_\matching =
  \setdef{\allocation^{b,*} \in \reals^n}{(M,\allocation^{b,*}) \text{
      is a balanced outcome}}$. Then every allocation in
  $\BB_\matching$ is locally stable under the
  dynamics~\eqref{eq:balance_dyn}.
\end{proposition}
\begin{proof}
  Take an arbitrary balanced allocation $\allocation^{b,*} \in
  \BB_\matching$ and consider the change of coordinates
  $\tilde{\allocation}^b = \allocation^b - \allocation^{b,*}$. Then
  \begin{align*}
    \dot{\tilde{\allocation}}^b &=
    -e^b(\tilde{\allocation}^b+\allocation^{b,*}).
  \end{align*}
  For brevity, denote this dynamics $\tilde{F} : \reals^n \rightarrow
  \reals^n$. We compute the Lie derivative of
  \begin{align*}
    V(\tilde{\allocation}^b) = \frac{1}{2}\max_{i \in
      \vertices}(\tilde{\allocation}^b_i)^2,
  \end{align*}
  along $\tilde{F}$. The derivation is very similar the one used in
  the proof of Proposition~\ref{prop:bounded},
  \begin{align*}
    \mathcal{L}_{\tilde{F}}V(\tilde{\allocation}^b) = \bigg \{a \in
    \reals &: a = - \hspace{-3mm} \sum_{i \in
      \mathbbm{M}(\tilde{\allocation}^b)} \hspace{-3mm} \lambda_i
    \tilde{\allocation}^b_i e_i^b(\tilde{\allocation}^b +
    \tilde{\allocation}^{b,*}),
    \\
    & \quad \text{for all $\lambda \in \reals^n_{\ge0}$
      s.t. \hspace{-3mm} $\sum_{i \in
        \mathbbm{M}(\tilde{\allocation}^b)} \hspace{-3mm} \lambda_i =
      1$} \bigg\},
  \end{align*}
  where
  \begin{align*}
    \mathbbm{M}(\tilde{\allocation}^b) := \frac{1}{2} \argmax_{i \in
      \vertices} (\tilde{\allocation}^b_i)^2.
  \end{align*}
  Consider one of the specific summands $-\tilde{\allocation}^b_i
  e_i^b(\tilde{\allocation}^b+\allocation^{b,*})$ for some $i \in
  \mathbbm{M}(\tilde{\allocation}^b)$. For $(i,j) \in \matching$, take
  $\tau \in \nextbest{i}{j}(\tilde{\allocation}^b +
  \allocation^{b,*})$ and $\kappa \in
  \nextbest{j}{i}(\tilde{\allocation}^b + \allocation^{b,*})$ so that
  we can write,  
  \begin{align*} -\tilde{\allocation}^b_i
      e_i^b(\tilde{\allocation}^b + \tilde{\allocation}^{b,*}) &=
      -\tilde{\allocation}^b_i (\tilde{\allocation}^b_i +
      \allocation^{b,*}_{i} - \tfrac{1}{2}(\weight_{i,j} + w_{i,\tau}
      - \tilde{\allocation}^b_{\tau}
      \\
      & \hspace{12mm} - \allocation^{b,*}_{\tau} - w_{j,\kappa} +
      \tilde{\allocation}^b_{\kappa} + \allocation^{b,*}_{\kappa}) ).
  \end{align*}
  Now, according to Lemma~\ref{lem:nbn}, there exists $\epsilon > 0$
  such that, for all $(k,l) \in \edges$, we have
  \begin{align*}
    \nextbest{k}{l}(\tilde{\allocation}^b+\allocation^{b,*}) & =
    \nextbest{k}{l}(\allocation^b) \subseteq
    \nextbest{k}{l}(\allocation^{b,*}) ,
  \end{align*}
  for all $\allocation^b$ such that $\|\tilde{\allocation}^b\| =
  \|\allocation^b - \allocation^{b,*}\| < \epsilon$. Therefore, for
  such allocations, we have $\tau \in
  \nextbest{i}{j}(\allocation^{b,*})$ and $\kappa \in
  \nextbest{j}{i}(\allocation^{b,*})$, and hence
  \begin{align*}
    -&\tilde{\allocation}^b_i e_i^b(\tilde{\allocation}^b +
    \tilde{\allocation}^{b,*})
    \\
    &= -\tilde{\allocation}^b_i(\tilde{\allocation}^b_i + \tfrac{1}{2}
    \tilde{\allocation}^b_{\tau} - \tfrac{1}{2}
    \tilde{\allocation}^b_{\kappa} + e^b_i(\allocation^{b,*}))
    \\
    &= -(\tilde{\allocation}^b_i)^2 - \tfrac{1}{2}
    \tilde{\allocation}^b_i\tilde{\allocation}^b_{\tau} + \tfrac{1}{2}
    \tilde{\allocation}^b_i\tilde{\allocation}^b_{\kappa}
    \\
    &\le -(\tilde{\allocation}^b_i)^2 + \tfrac{1}{4}
    (\tilde{\allocation}^b_i)^2 +
    \tfrac{1}{4}(\tilde{\allocation}^b_{\tau})^2 + \tfrac{1}{4}
    (\tilde{\allocation}^b_i)^2 +
    \tfrac{1}{4}(\tilde{\allocation}^b_{\kappa})^2 \le 0,
  \end{align*}
  where we have used the fact that $\allocation^{b,*} \in
  \BB_\matching$ in the second equality, the inequality $ab \le
  \tfrac{1}{2}a^2 + \tfrac{1}{2}b^2$ for $a,b \in \reals$ in the first
  inequality and the fact that $i \in
  \mathbbm{M}(\tilde{\allocation}^b)$ in the last inequality. Thus $a
  \le 0$ for each $a \in
  \mathcal{L}_{\tilde{F}}\tilde{V}(\tilde{\allocation}^b)$ when
  $\|\tilde{\allocation}^b\| \le \epsilon$, which means that
  $\tilde{\allocation}^b = 0$ is locally stable. In the original
  coordinates, $\allocation^b = \allocation^{b,*}$ is locally
  stable. Since $\allocation^{b,*}$ is arbitrary, we deduce that every
  allocation in $\BB_\matching$ is locally stable.
\end{proof}

The boundedness of the balancing errors together with the local
stability of the balanced allocations under the dynamics allow us to
employ the LaSalle Invariance Principle, cf. Theorem~\ref{th:lasalle}
in the proof of the next result and establish the pointwise
convergence of the dynamics to an allocation in a balanced outcome
with matching~$\matching$.

\begin{proposition}\longthmtitle{Convergence to a balanced
    outcome}\label{prop:balanced}
  Given a matching $\matching$, let $t \rightarrow \allocation^b(t)$
  be a trajectory of~\eqref{eq:stable_dyn} starting from an initial
  point in $\reals^n$. Then $t \mapsto (\matching,\allocation^b(t))$
  converges to a balanced outcome.  Moreover, the
  dynamics~\eqref{eq:balance_dyn} is distributed with respect to
  $2$-hop neighborhoods over $\graph$.
\end{proposition}
\begin{proof}
  Note that, for each pair of matched agents $(i,j) \in \matching$,
  the sum $\dot{\allocation}^b_i + \dot{\allocation}^b_j =
  \weight_{i,j} - (\allocation^b_i + \allocation^b_j)$, implying that
  $\allocation^b_i(t) + \allocation^b_j(t) \rightarrow \weight_{i,j}$
  exponentially fast. For each unmatched agent $k$, one has that
  $\dot{\allocation}^b_k = -\allocation^b_k$, implying that
  $\allocation^b_i(t) \rightarrow 0$ exponentially fast. Therefore, it
  follows that $t \mapsto (\matching,\allocation^b(t))$ converges to
  the set of (valid) outcomes. It remains to further show that it
  converges to the set of balanced outcomes. Following the approach
  employed in the proof of Proposition~\ref{prop:bounded}, we argue
  with the trajectories of~\eqref{eq:xi-omega-dyn}, which we showed
  contain the trajectory $t \mapsto e^b(\allocation^b(t))$. For
  matched agents $(i,j) \in \matching$,
  \setcounter{equation}{17}
  \begin{align}\label{eq:sum-xi}
    \dot{\xi}_i + \dot{\xi}_j = - (\xi_i + \xi_j),
  \end{align}
  under the dynamics~\eqref{eq:xi-omega-dyn}. Interestingly, this
  dynamics is independent of $\omega$. Thus, using the Lyapunov
  function
  \begin{align*}
    \widetilde{V}(\xi) = \tfrac{1}{2}\sum_{(i,j) \in \matching}(\xi_i
    + \xi_j)^2 + \tfrac{1}{2}\hspace{-4mm}\sum_{\substack{\{i \in
        \vertices :\\ \text{$i$ is
          unmatched}\}}}\hspace{-4mm}(\xi_i)^2,
  \end{align*}
  it is trivial to see that
  \begin{align*}
    \mathcal{L}_{\mathcal{F}[F]}\widetilde{V}(\xi) = -\sum_{(i,j)
      \in \matching}(\xi_i + \xi_j)^2 -
    \hspace{-4mm}\sum_{\substack{\{i \in \vertices :\\ \text{$i$ is
          unmatched}\}}}\hspace{-4mm}(\xi_i)^2,
  \end{align*}
  which, again, is independent of $\omega$. By the boundedness of $t
  \mapsto \xi(t)$ established in Proposition~\ref{prop:bounded}, and
  using $\widetilde{V}$, we are now able to apply the LaSalle
  Invariance Principle, cf. Theorem~\ref{th:lasalle}, which asserts
  that the trajectory $t \mapsto \xi(t)$ converges to the largest
  weakly positively invariant set $\mathcal{M}$ contained in
  \begin{align*}
    L :&= \{ \xi \in \reals^n :
    \mathcal{L}_{\mathcal{F}[F]}\widetilde{V}(\xi) = 0 \} \\ &= \{
    \xi \in \reals^n : \xi_i = -\xi_j, \; \forall (i,j) \in \matching,
    \\ &\hspace{21.5mm} \text{and } \xi_i = 0 \text{ if $i$ is
      unmatched} \} .
  \end{align*}
  Incidentally, this set is closed already which is why we omit the
  closure operator. We next show, using the fact that $t \mapsto
  V(\xi(t))$ is non-increasing (cf. Proposition~\ref{prop:bounded})
  and the weak invariance of $\mathcal{M}$, that in fact $\mathcal{M}
  = \{0\}$.  Take a point $\xi \in \mathcal{M} \subseteq L$ and take
  an $i \in \maxerrors(\xi)$. If $i$ is unmatched, then $\xi_i = 0$
  already and the proof would be complete. So, assume $(i,j) \in
  \matching$ for some $j \in \vertices$. Then, $\xi_j = -\xi_i$ and it
  also holds that $\dot{\xi}_i = -\dot{\xi}_j$ (see
  e.g.,~\eqref{eq:sum-xi}). In fact, it must be that $\dot{\xi}_i =
  \dot{\xi}_j = 0$, otherwise one of $\xi_i$ or $\xi_j$ would be
  increasing, which would contradict $t \mapsto V(\xi(t))$ being
  non-increasing. If $\nextbest{i}{j}(\omega) =
  \nextbest{j}{i}(\omega) = \emptyset$ then $0 = \dot{\xi}_i = -\xi_i
  = \xi_j$, which would complete the proof. Suppose then that $\tau =
  \nextbest{i}{j}(\omega)$ and $\nextbest{j}{i}(\omega) =
  \emptyset$. Then $0 = \dot{\xi}_i = -\xi_i +
  \tfrac{1}{2}\xi_{\tau}$, which contradicts $i \in \maxerrors(\xi)$
  (unless of course $\xi_i = 0$, which would complete the proof). A
  similar argument holds if $\nextbest{i}{j}(\omega) = \emptyset$ and
  $\nextbest{j}{i}(\omega) = \kappa$. The final case is if
  $\nextbest{i}{j}(\omega) = \tau$ and $\nextbest{j}{i}(\omega) =
  \kappa$. In this case, $0 = \dot{\xi}_i = -\xi_i +
  \tfrac{1}{2}\xi_{\tau} - \tfrac{1}{2}\xi_{\kappa}$. So as not to
  contradict $i \in \maxerrors(\xi)$, it must be that $\xi_i =
  -\xi_{\tau} = \xi_{\kappa}$, which means that $\tau,\kappa \in
  \maxerrors(\xi)$ as well. Therefore, using the same argument we used
  for $i$, it must be that $0 = \dot{\xi}_{\tau} =
  \dot{\xi}_{\kappa}$. Assume without loss of generality that
  $\xi_{\tau}$ is strictly negative (if it were zero the proof would
  be complete and if it were positive we could argue instead with
  $\xi_{\kappa}$). This means that $\xi_{\tau}$ grows larger at a
  constant rate since $\dot{\omega}_{\tau} = -\xi_{\tau}$. At some
  time, it would happen that $\omega_{\tau} > w_{i,\tau}$, which would
  make $\nextbest{i}{j}(\omega) = \emptyset$. This corresponds to a
  case we previously considered where we showed that, so as not to
  contradict the monotonicity of $t \mapsto V(\xi(t))$ it must be that
  $\xi_i = 0$. In summary, $\mathcal{M} = \{0\} \subset \reals^n$, so
  $\xi(t) \rightarrow 0$. By construction of the
  dynamics~\eqref{eq:xi-omega-dyn} it follows that
  $e^b(\allocation^b(t)) \rightarrow 0$ which means, by construction
  of $e^b$, that $(M,\allocation^b(t))$ converges to the set of
  balanced outcomes. This, along with the local stability of each
  balanced allocation (cf. Proposition~\ref{prop:local_stability}) is
  sufficient to ensure pointwise convergence to a balanced
  outcome~\cite[Proposition 2.2]{QH-WMH:09}. Finally, it is clear
  from~\eqref{eq:balance_dyn} that the dynamics is distributed with
  respect to 2-hop neighborhoods, which completes the proof.
\end{proof}

\myclearpage
\section{Distributed dynamics to find Nash outcomes}\label{sec:nash}

In this section, we combine the previous developments to propose
distributed dynamics that converge to Nash outcomes. The design of
this dynamics is inspired by the following result
from~\cite{YA-BB-LC-ND-YP:09} revealing that balanced outcomes
associated with maximum weight matchings are stable.

\begin{proposition}\longthmtitle{Balanced implies
    stable}\label{prop:quasi} 
  Let $\matching$ be a maximum weight matching on $\graph$ and suppose
  that $\graph$ admits a stable outcome. Then, a balanced outcome of
  the form $(\matching,\allocation^b)$ is also stable, and thus Nash.
\end{proposition}

In a nutshell, our proposed dynamics combine the fact that (i) the
distributed dynamics~\eqref{eq:stable_dyn} of Section~\ref{sec:stable}
allow agents to determine a maximum weight matching and (ii) given
such a maximum weight matching, the distributed
dynamics~\eqref{eq:balance_dyn} of Section~\ref{sec:balance} converge
to balanced outcomes. The combination of these facts with
Proposition~\ref{prop:quasi} yields the desired convergence to Nash
outcomes.  

When putting the two dynamics together, however, one should note that
the convergence of~\eqref{eq:stable_dyn} is asymptotic, and hence
agents implement~\eqref{eq:balance_dyn} before the final stable
matching is realized. To do this, we have agents guess with whom (if
any) they will be matched in the final Nash outcome. An agent $i$
guesses that it will match with $j \in \neigh(i)$ if the current value
of the matching state $m_{i,j}(t)$ coming from the
dynamics~\eqref{eq:stable_dyn} is closest to $1$ as compared to all
other neighbors in $\neigh(i) \setminus j$. As we show later, this
guess becomes correct in finite time. Formally, agent $i$ predicts its
\emph{partner} by computing
\begin{align*}
  \partner_i(m) \! = \! \{ j \in \neigh(i) \! : \! |m_{i,j}-1| \! < \!
  |m_{i,k}-1|, \forall k \in \neigh(i) \! \setminus \! j\}.
\end{align*}
Clearly, $\partner_i(m)$ is at most a singleton and can be computed
by~$i$ using local information. If $\partner_i(m) = \{j\}$, we use the
slight abuse of notation and write $\partner_i(m) = j$. 

With the above discussion in mind, we next propose the following
distributed strategy: each agent $i \in \vertices$ implements its
corresponding dynamics in~\eqref{eq:stable_dyn} to find a stable
outcome but only begins balancing its allocation if, for some $j \in
\neigh(i)$, agents $i$ and $j$ identify each other as
partners. Formally, this dynamics is represented by, for each $i \in
\vertices$,
%
\begin{subequations}\label{eq:nash_dyn}
  \begin{align}
    \dot{\allocation}^s_i &= \left\{
      \begin{alignedat}{2}
        &\nomflow^{\allocation}_i(\allocation^s,\slack,m), && \;
        \allocation^s_i > 0,
        \\
        &\max\{0,\nomflow^{\allocation}_i(\allocation^s,\slack,m)\},
        && \; \allocation^s_i = 0,
      \end{alignedat} \right. \label{eq:nash_stable1}
    \\
    \dot{\allocation}^b_i &= \left\{ \begin{alignedat}{2}
        &-e^b_i(\allocation^b), &&\; \; \text{if for some } j \in
        \neigh(i),
        \\
        & && \; \; \partner_i(m) = j \text{ and } \partner_j(m) = i,
        \\
        &-\allocation^b_i, && \; \; \text{otherwise,}
      \end{alignedat} \right. \label{eq:nash_balance}
  \end{align}
  and, for each $j \in \neigh^+(i)$,
  \begin{align}
    \dot{\slack}_{i,j} &= \left \{ \begin{alignedat}{2}
        &\nomflow^{\slack}_{i,j}(\allocation^s,\slack,m), && \;
        \slack_{i,j} > 0,
        \\
        &\max\{0,\nomflow^{\slack}_{i,j}(\allocation^s,\slack,m)\}, &&
        \; \slack_{i,j} = 0,
      \end{alignedat} \right. \label{eq:nash_stable2}
    \\
    \dot{m}_{i,j} &= \allocation^s_i + \allocation^s_j - \slack_{i,j}
    - \weight_{i,j}. \label{eq:nash_stable3}
  \end{align}
\end{subequations}
The state of agent $i \in \vertices$ is then
\begin{align*}
  (\allocation^s_i,\allocation^b_i,\{s_{i,j}&\}_{j \in
    \neigh^+(i)},\{m_{i,j}\}_{j \in \neigh^+(i)})
  \\
  &\in \realsnonnegative \times \reals \times
  \reals^{|\neigh^+(i)|}_{\ge 0} \times \reals^{|\neigh^+(i)|}.
\end{align*}
For convenience, we denote the dynamics~\eqref{eq:nash_dyn} by
\begin{align*}
  F^{\operatorname{Nash}} \! : \! \reals^n_{\ge 0} \! \times \!
  \reals^n \! \times \!  \reals^{|\edges|}_{\ge 0} \! \times \!
  \reals^{|\edges|} \rightarrow \reals^n_{\ge 0} \! \times \reals^n \!
  \times \! \reals^{|\edges|}_{\ge 0} \!  \times \! \reals^{|\edges|}.
\end{align*}
The dynamics~\eqref{eq:nash_dyn} can be viewed as a cascade system,
with the states $m$ feeding into the balancing
dynamics~\eqref{eq:nash_balance}. The next result establishes the
asymptotic convergence of this cascade system.

\begin{theorem}\longthmtitle{Asymptotic convergence to Nash
    outcomes}\label{th:nash} Let $t \rightarrow
  (\allocation^s(t),\allocation^b(t),\slack(t),m(t))$ be a trajectory
  of~\eqref{eq:nash_dyn} starting from an initial point in
  $\reals^n_{\ge 0} \times \reals^n \times
  \realsnonnegative^{|\edges|} \times \reals^{|\edges|}$. Then, if
  there exists a stable outcome, for some $T > 0$ the maximum weight
  matching $\matching$ is well-defined by the implication
  \begin{align*}
    (i,j) \in \matching \quad \Leftrightarrow \quad 
    \partner_i(m(t)) = j \text{ and } \partner_j(m(t)) = i.
  \end{align*}
  for all $t \ge T$. Furthermore, $t \mapsto (M,\allocation^{b}(t))$
  converges to a Nash outcome. Moreover,~\eqref{eq:nash_dyn} is
  distributed with respect to $2$-hop neighborhoods over $\graph$.
\end{theorem}
\begin{proof}
  Let $m^* \in \reals^{|\edges|}$ be the unique integral solution
  of~\eqref{eq:max_match}.  The asymptotic convergence properties
  of~\eqref{eq:stable_dyn}, cf. Proposition~\ref{prop:stab}, guarantee
  that, for every $\epsilon > 0$, there exists $T > 0$ such that, for
  all $t \ge T$,
  \begin{align*}
    \epsilon >
    \begin{cases}
      |m_{i,j}(t)-1|, \quad &\text{if $m^*_{i,j} = 1$,}
      \\
      |m_{i,j}(t)|, \quad &\text{if $m_{i,j}^* = 0$.}
    \end{cases}
  \end{align*}
  Thus, taking $\epsilon < \tfrac{1}{2}$, it is straightforward to see
  that the matching induced by the implication
  \begin{align*}
    (i,j) \in \matching \quad \Leftrightarrow \quad \partner_i(m(t)) =
    j \text{ and } \partner_j(m(t)) = i,
  \end{align*}
  is well-defined, a maximum weight matching, and constant for all $t
  \ge T$. Then, considering only $t \ge T$ and applying
  Propositions~\ref{prop:balanced} and~\ref{prop:quasi}, we deduce
  that $t \mapsto (M,\allocation^b(t))$ converges to a Nash
  outcome. The fact that~\eqref{eq:nash_dyn} is distributed with
  respect to $2$-hop neighborhoods follows from its definition, which
  completes the proof.
\end{proof}

Finally, we comment on the robustness properties of the Nash
bargaining dynamics~\eqref{eq:nash_dyn} against perturbations such as
communication noise, measurement error, modeling uncertainties, or
disturbances. A central motivation for using the linear programming
dynamics~\eqref{eq:disc_dyn}, and continuous-time dynamics in general,
is that there exist various established robustness characterizations
for them. In particular, using previously established results
from~\cite{DR-JC:13-tac}, it holds that~\eqref{eq:nash_dyn} is a
`well-posed' dynamics, as defined in~\cite{RG-RGS-ART:09}. As a
straightforward consequence of~\cite[Theorem 7.21]{RG-RGS-ART:09}, the
Nash bargaining dynamics is robust to small perturbations, as we state
next.

\begin{corollary}\longthmtitle{Robustness to small
    perturbations}\label{cor:robustness}
  Given a graph $\graph$, assume there exists a stable outcome and let
  $t \rightarrow (\allocation^s(t),\allocation^b(t),\slack(t),m(t))$
  be a trajectory, starting from an initial point in $\reals^n_{\ge 0}
  \times \reals^n \times \reals_{\ge 0}^{|\edges|} \times
  \reals^{|\edges|}$, of the perturbed dynamics
  \begin{align*}
    (\dot{\allocation}^s,\dot{\allocation}^b,\dot{\slack},\dot{m}) =
    F^{\operatorname{Nash}}(\allocation^s \! + \! d_1,\allocation^b \!
    + \! d_2,\slack \! + \! d_3,m \! + \! d_4) \! + \! d_5
  \end{align*}
  where $d_1,d_2: \realsnonnegative \mapsto \reals^{n}$, $d_3,d_4 :
  \realsnonnegative \mapsto \reals^{|\edges|}$, and $d_5 :
  \realsnonnegative \mapsto \reals^{2n+2|\edges|}$ are
  disturbances. Then, for every $\epsilon > 0$, there exist $\delta,T
  > 0$ such that, for $\max_{i}\NormInf{d_i} < \delta$, the maximum
  weight matching $\matching$ is well-defined by the implication
  \begin{align*}
    (i,j) \in \matching \quad \Leftrightarrow \quad
    \partner_i(m(t)) = j \text{ and } \partner_j(m(t)) = i,
  \end{align*}
  for all $t \ge T$, and $t \mapsto (\matching, \allocation^b(t))$
  converges to an $\epsilon$-neighborhood of the set of Nash outcomes
  of~$\graph$.
\end{corollary}

\begin{figure}[hbt!]
  \centering
  \includegraphics[width=0.75\linewidth]{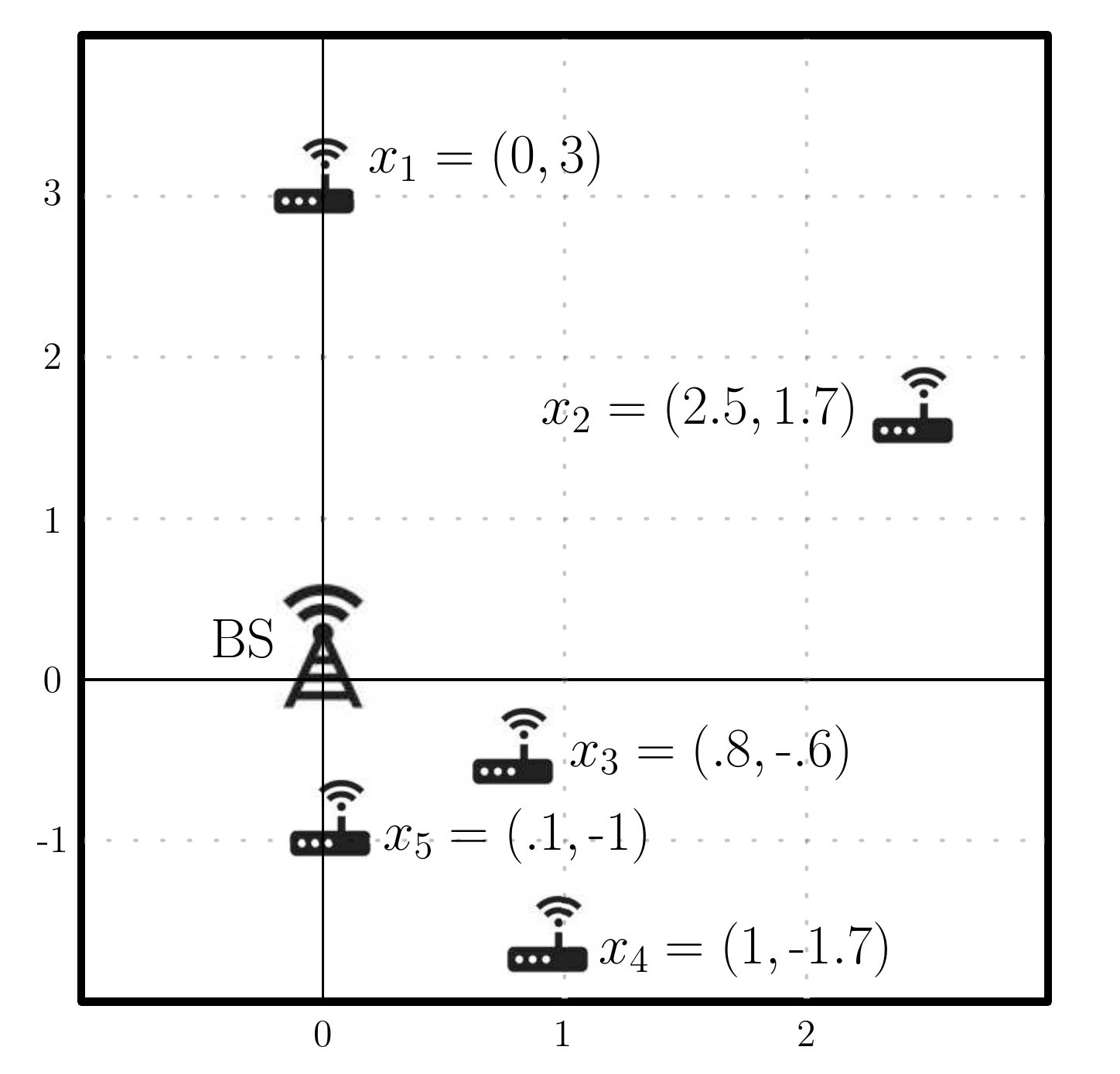}
  \caption{Spatial distribution of devices $\{1,\dots,5\}$ and the
    base station (BS).}\label{fig:distribution_of_devices}
\end{figure}

\begin{figure}[hbt!]
  \centering
  \includegraphics[width=0.95\linewidth]{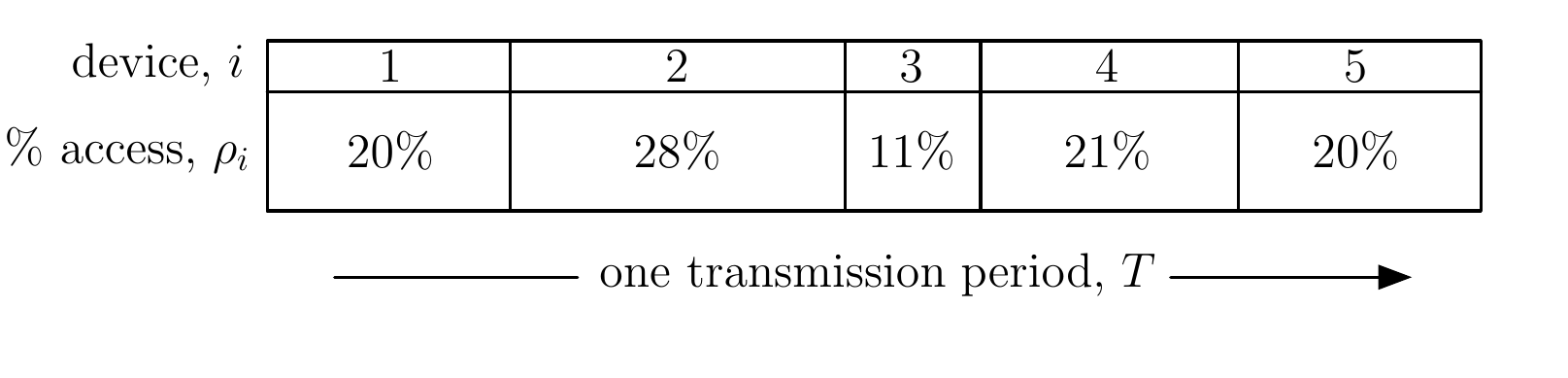}
  \caption{TDMA transmission time allocations for each device.}\label{fig:TDMA}
\end{figure}

\begin{figure}[hbt!]
  \centering
  \includegraphics[width=0.55\linewidth]{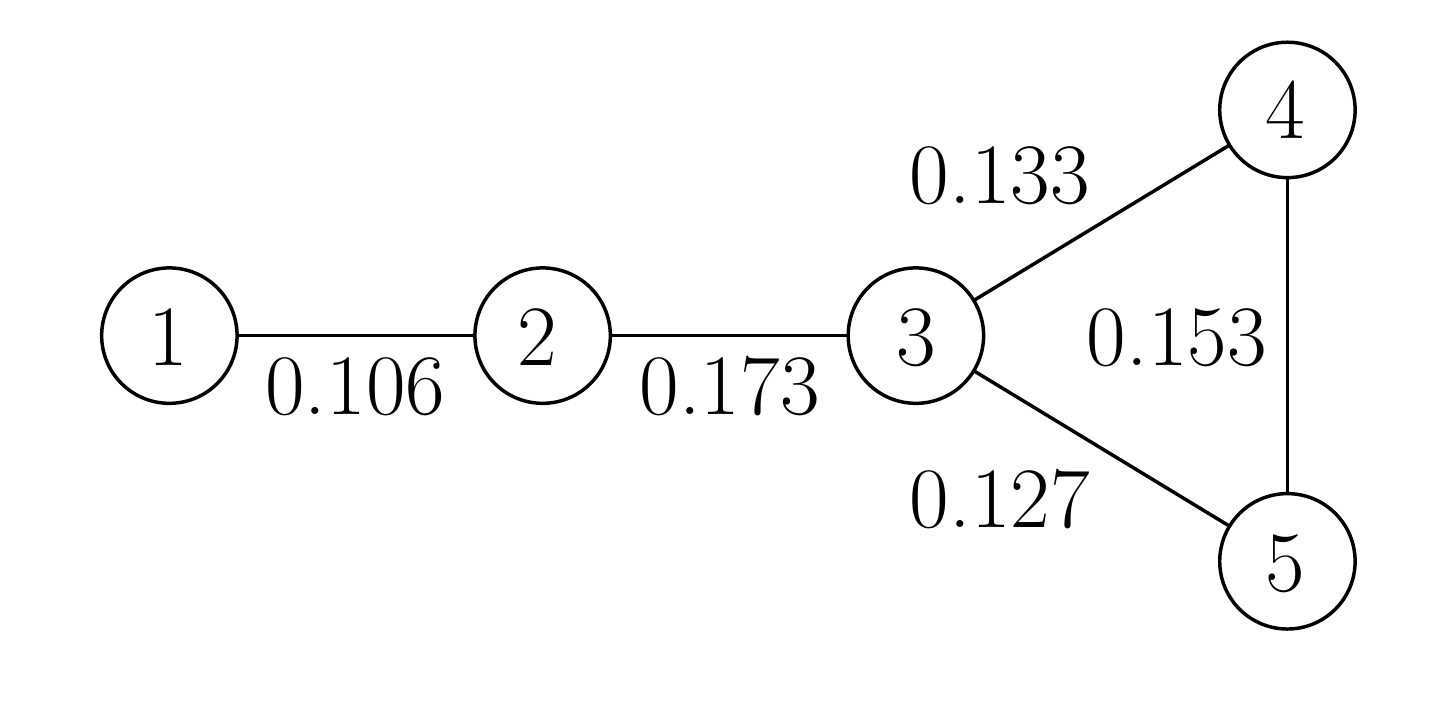}
  \caption{The bargaining graph resulting from the position and TDMA
    transmission time allocations for each device. Here, we have taken
    $P_{\max} = 3$.}\label{fig:bargain_graph}
\end{figure}

\myclearpage
\section{Application to multi-user wireless
  communication} \label{sec:sims}

In this section, we provide some simulation results of our proposed
Nash bargaining dynamics as applied to a multi-user wireless
communication scenario. The scenario we describe here is a simplified
version of the one found in~\cite{WS-ZH-MD-AH:09}, and we direct the
reader to that reference for a more detailed discussion on the
model. We assume that there are $n = 5$ single antenna devices
distributed spatially in an environment that send data to a fixed base
station. We denote the position of device $i \in \{1,\dots,5\}$ as
$x_i \in \reals^2$ and we assume without loss of generality that the
base station is located at the
origin. Figure~\ref{fig:distribution_of_devices} illustrates the
position of the devices. An individual device's transmission is
managed using a time division multiple access (TDMA) protocol. That
is, each device $i$ is assigned a certain percentage $\rho_i$ of a
transmission period of length $T$ in which it is allowed to transmit
as specified in Figure~\ref{fig:TDMA}. We use a commonly used model
for the capacity $c_i > 0$ of the communication channel from device
$i$ to the base station, which is a function of their relative
distance,
\begin{align*}
  c_i = \log( 1 + |x_i|^{-1}).
\end{align*}
In the above, we have taken various physical parameters (such as
transmit power constraints, path loss constants, and others) to be $1$
for the sake of presentation. Since $i$ only transmits for $\rho_i$
percent of each transmission period, the effective capacity of the
channel from device $i$ to base station is $\rho_i c_i$. It is
well-known in wireless communication~\cite{CW-XH-XG-GZ-JT:10} that
multiple antenna devices can improve the channel capacity. Thus,
devices $i$ and $j$ may decide to share their data and transmit a
multiplexed data signal in both $i$ and $j$'s allocated time slots. In
essence, $i$ and $j$ would behave as a single virtual $2$-antenna
device. The resulting channel capacity is given by
\begin{align*}
    c_{i,j} = \log( 1 + |x_i|^{-1} + |x_j|^{-1}),
\end{align*}
\begin{figure*}[hbt!]
  \centering \subfigure[Device
    allocations]{\includegraphics[width=0.4\linewidth]{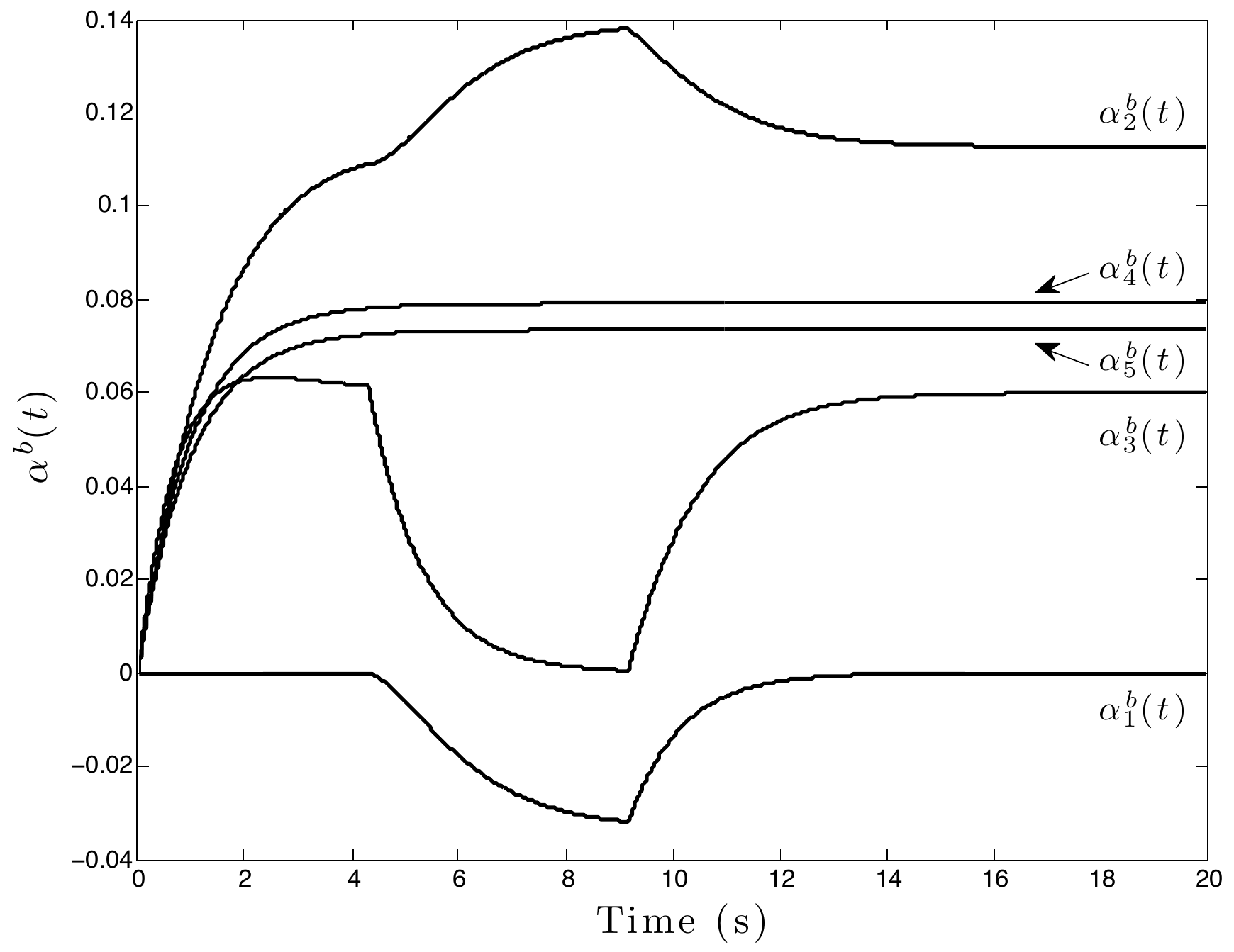}} \hspace{10mm}
  \subfigure[Matching
    states]{\includegraphics[width=0.4\linewidth]{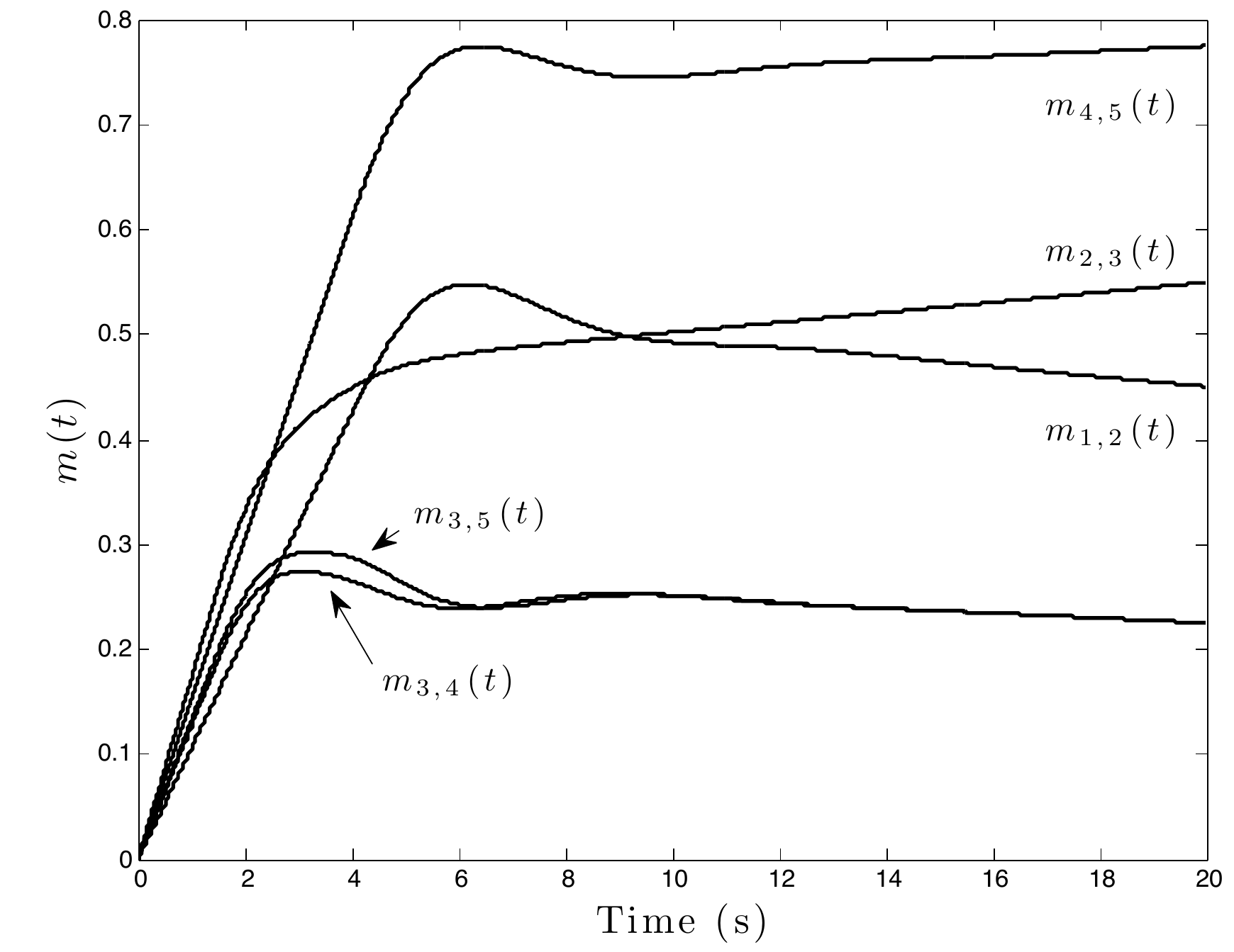}}
  \caption{Evolution of each device's allocation and the matching
    states in dynamics~\eqref{eq:nash_dyn}. At various times (i.e., $t
    \approx 4$ and $9$), certain devices change who they identify as
    partners in the matching which explains the kinks in the
    trajectories at those times. This occurs because of the evolution
    of the matching states in (b) and devices cannot correctly deduce
    the stable matching until $t \approx 9$. The final convergence of
    the matching states to $\{0,1\}^{|\edges|}$ (which we do not show
    for the sake of presentation) takes much longer than devices need
    to accurately identify a Nash outcome.
    }\label{fig:allocations_and_matching}
\end{figure*}
which is greater than both $c_i$ and $c_j$. However, there is a cost
to agent $i$ and $j$ cooperating in this way because their data must
be transmitted to each other. We assume that the device-to-device
transmissions do not interfere with the device-to-base station
transmissions. The power needed to transmit between $i$ and $j$ is
given by
\begin{align*}
  P_{i,j} = |x_i-x_j|.
\end{align*}
If this power is larger than some $P_{\max} > 0$, then $i$ and $j$
will not share their data. We can model this scenario via a graph
$\graph = (\vertices,\edges,\weights)$, where $\vertices =
\{1,\dots,5\}$ are the devices, edges correspond to whether or not $i$
and $j$ are willing, based on the power requirements, to share their
data
\begin{align*}
  (i,j) \in \edges \Leftrightarrow P_{i,j} \le P_{\max},
\end{align*}
and the edge weights represent the increase in effective channel
capacity should devices cooperate,
\begin{align*}
  w_{i,j} = (\rho_i+\rho_j)c_{i,j} - \rho_ic_i - \rho_jc_j, \quad
  \forall (i,j) \in \edges.
\end{align*}
Figure~\ref{fig:bargain_graph} shows this graph, using the data for
the scenario we consider. It is interesting to note that, besides
channel capacity and power constraints, one could incorporate other
factors into the edge weight definition. For example, if privacy is a
concern in the network, then devices may be less likely to share their
data with untrustworthy devices which can be modeled by a smaller edge
weight.

\begin{figure}[hbt!]
  \centering
  \includegraphics[width=0.7\linewidth]{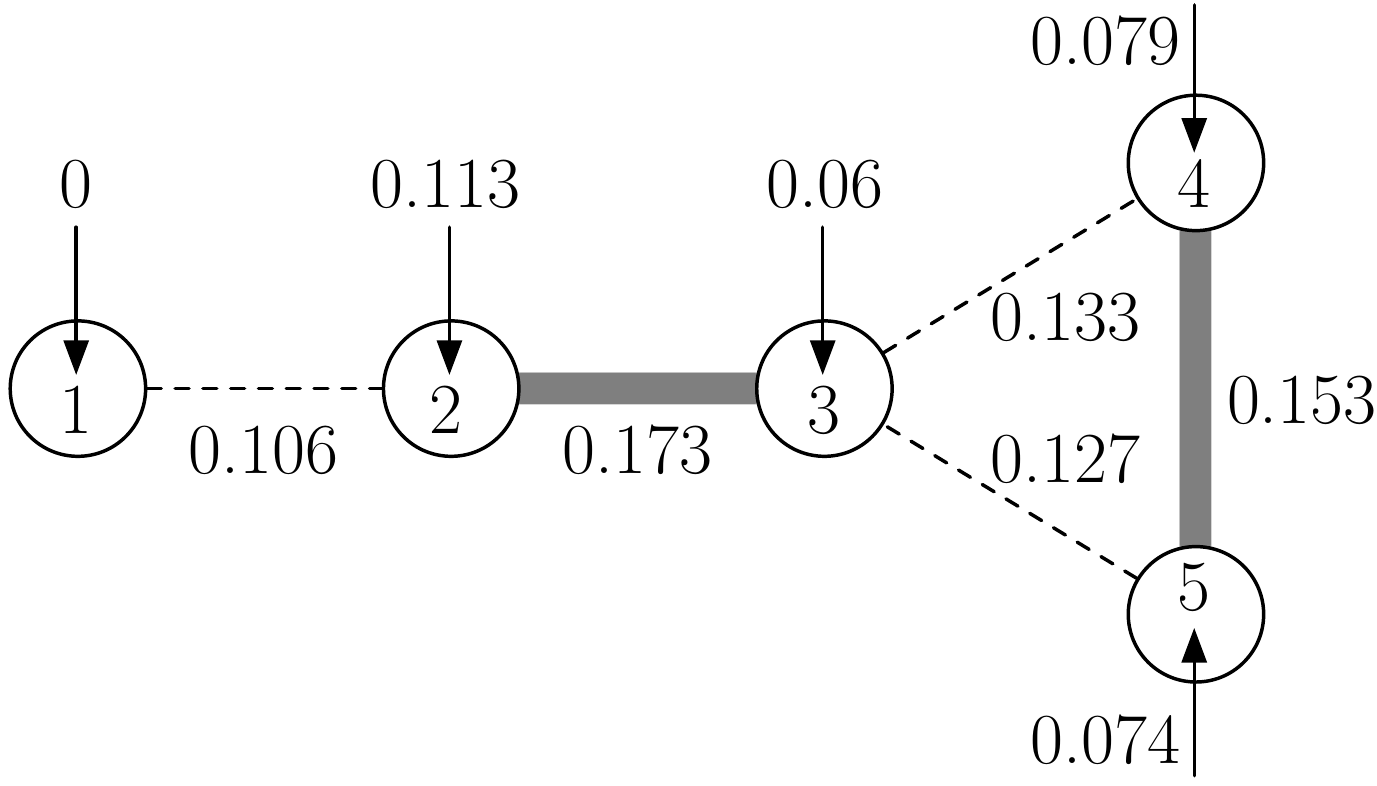}
  \caption{Nash outcome that is distributedly computed by
    devices. Device matchings are shown by thicker grey edges and
    allocations to each device are indicated with
    arrows}\label{fig:nash_graph}
\end{figure}

\begin{table}[hbt!]
  \caption{Improvements in capacity due to
    collaboration}\label{table:improvement}
  \begin{center}
  \begin{tabular}{| c | c | c | c |}
    \hline & Effective channel & Increase in effective & \% \\ Device
    & capacity without & channel capacity & improve- \\ &
    collaboration, $c_i$ & in Nash outcome, $\allocation^b_i$ &
    ment \\ \hline 1 & 0.288 & 0 & 0 \\ \hline 2 & 0.288 & 0.113 &
    39.2 \\ \hline 3 & 0.693 & 0.06 & 8.7 \\ \hline 4 & 0.693 & 0.079
    & 11.4 \\ \hline 5 & 0.406 & 0.074 & 18.2 \\ \hline \hline
    \{1,\dots,5\} & 0.441 & 0.070 & 15.8 \\ \hline
  \end{tabular}
  \end{center}
\end{table}

A matching $M$ in the context of this setting corresponds to disjoint
pairs of devices that decide to share their data and transmission time
slots in order to achieve a higher effective channel capacity. An
allocation corresponds to how the resulting improved bit rate is
divided between matched devices. For example, if $i$ is allocated an
amount of $\allocation^b_i$, then $i$ and $j$ will transmit their data
such that $i$'s data reaches the base station at a rate of $c_i +
\allocation^b_i$. The percent improvement in bit rate for $i$ is then
given by~${\allocation^b_i}/{c_i}$. Devices use the
dynamics~\eqref{eq:nash_dyn} to find, in a distributed way, a Nash
outcome for this problem. Figure~\ref{fig:allocations_and_matching}
reveals the resulting state trajectories and
Figure~\ref{fig:nash_graph} displays the final Nash outcome. The
percent improvements resulting from collaboration for each device are
collected in Table~\ref{table:improvement}. The last row in this table
show that the network-wide improvement is $15.8\%$. Before bargaining,
devices $1$ and $2$ have the lowest individual channel capacities and
would thus greatly benefit from collaboration. However, due to power
constraints, device $1$ can only match with device $2$, who in turn
prefers to match with device $3$. This explains why, in the end,
device $1$ is left
unmatched. Figure~\ref{fig:noisy_balanced_allocations} illustrates how
convergence is still achieved when noise is present in the devices'
dynamics, as forecasted by Corollary~\ref{cor:robustness}.


\begin{figure}[hbt!]
  \centering
  \includegraphics[width=0.8\linewidth]{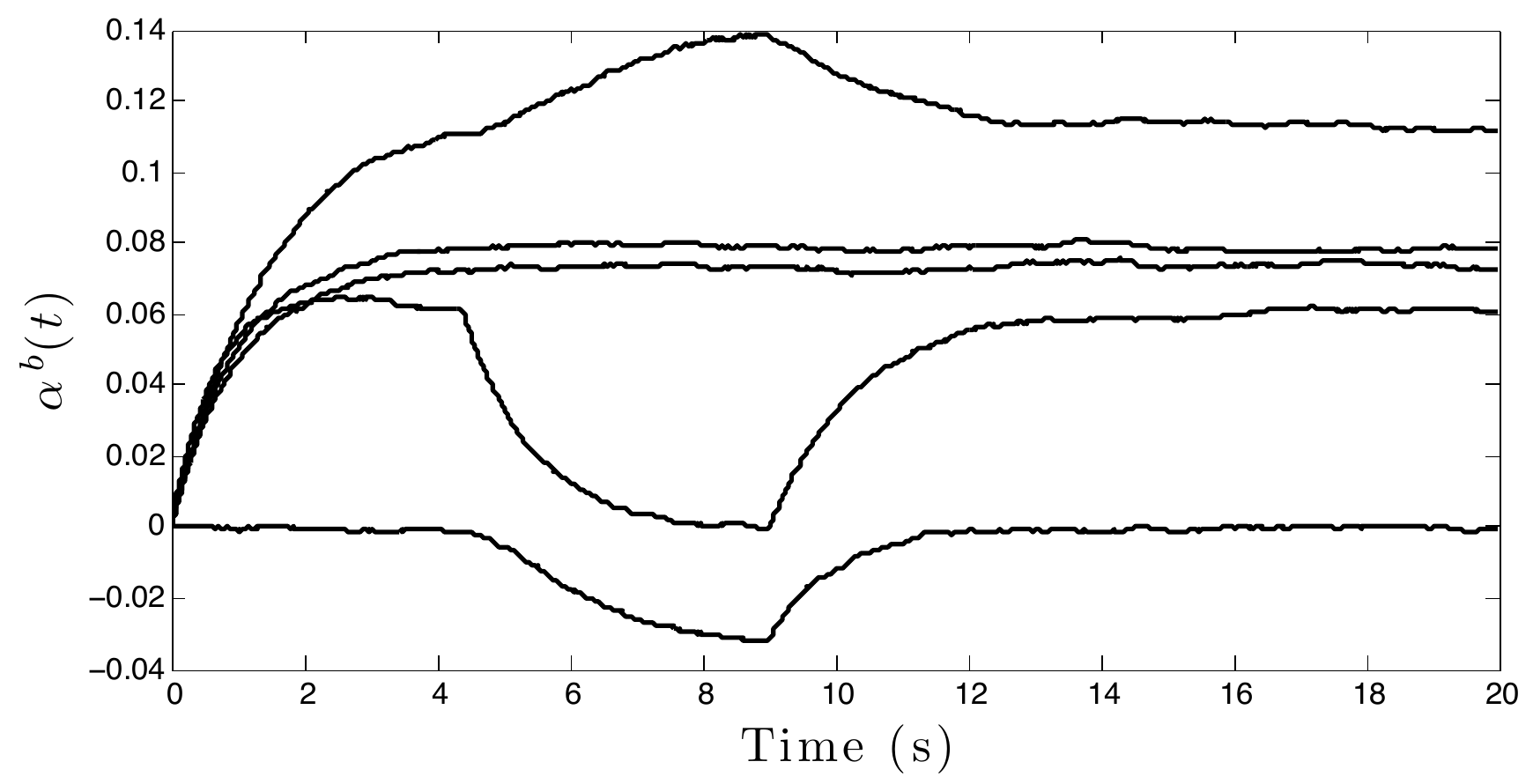}
  \caption{When the devices' dynamics are subjected to noise (normally
    distributed with zero mean and standard deviation~$0.01$), the
    stable matchings are still correctly deduced and devices'
    allocations converge to a neighborhood of the allocations in the
    Nash outcome.}\label{fig:noisy_balanced_allocations}
\end{figure}

\myclearpage
\section{Conclusions and future work}\label{sec:conclusions}

We have considered bargaining in dyadic-exchange networks, where
individual agents decide with whom (if any) to match and agree on an
allocation of a common good. For such scenarios, valid notions of
outcomes include stable, balanced, and Nash.  We have designed
provably correct distributed dynamics that asymptotically converge to
each of these classes of outcomes. Our technical approach combines
graph- and game-theoretic notions with techniques from set-valued
dynamics, stability theory, and distributed linear programming.  We
have illustrated the performance of the proposed coordination
algorithm in a wireless communication scenario, where we showed how
agent collaborations can, in a fair way, improve both individual and
network-wide performance.  Future work will include considering other
solution concepts on dyadic-exchange networks and applying our
techniques to multi-exchange networks (i.e., allowing coalitions of
more than two). In addition, we would like to study the rate of
convergence and establish more quantifiable robustness properties of
balancing dynamics; in particular, the effects of time delays,
adversarial agents, and dynamically changing system data. Finally, we
wish to apply our dynamics to other coordination tasks and implement
them on a multi-agent testbed.


\bibliographystyle{ieeetr}%
\bibliography{alias,Main,Main-add,JC}

\appendix

The following result is used in the proof of
Proposition~\ref{prop:local_stability} to establish the local
stability of each balanced allocation under the
dynamics~\eqref{eq:balance_dyn}.

\begin{lemma}\longthmtitle{Upper-semicontinuity of the
    next-best-neighbor sets map}\label{lem:nbn}
  Let $\allocation^{b,*} \in \reals^n$. Then there exists $\epsilon >
  0$ such that, for all $(i,j) \in \edges$ and all $\|\allocation^b -
  \allocation^{b,*}\| < \epsilon$, the following inclusion holds
  \begin{align*}
    \nextbest{i}{j}(\allocation^b) \subseteq
    \nextbest{i}{j}(\allocation^{b,*}).
  \end{align*}
\end{lemma}
\begin{proof}
  Note that, since the number of edges is finite, it is enough to
  prove that such $\epsilon$ exists for each edge $(i,j) \in \edges$
  (because then one takes the minimum over all of them).  Therefore,
  let $(i,j) \in \edges$ and, arguing by contradiction, assume that
  for every $\epsilon>0$, there exists $\allocation^b$ with
  $\|\allocation^b - \allocation^{b,*} \| < \epsilon$ such that $
  \nextbest{i}{j}(\allocation^b) \not \subseteq
  \nextbest{i}{j}(\allocation^{b,*})$. Equivalently, suppose that
  $\{\allocation^{b,k}\}_{k = 1}^{\infty}$ is a sequence converging to
  $\allocation^{b,*}$ such that, for every $k$, there exists a $\tau^k
  \in \nextbest{i}{j}(\allocation^{b,k}) \setminus
  \nextbest{i}{j}(\allocation^{b,*})$. By definition of the
  next-best-neighbor set, it must be that
  \begin{align*}
    w_{i,\tau^k} - \allocation^{b,k}_{\tau^k} \ge w_{i,\tau} -
    \allocation^{b,k}_{\tau},
  \end{align*}
  for all $\tau \in \neigh(i) \setminus j$. Since $\neigh(i) \setminus
  j$ has a finite number of elements, there must be some $\hat{\tau}
  \in \neigh(i)$ such that $\tau^k = \hat{\tau}$ infinitely
  often. Therefore, let $\{k_{\ell}\}_{\ell=1}^{\infty}$ be a
  subsequence such that $\tau^{k_{\ell}} = \hat{\tau}$ for all
  $\ell$. Then
  \begin{align*}
    w_{i,\hat{\tau}} - \allocation^{b,k_{\ell}}_{\hat{\tau}} \ge
    w_{i,\tau} - \allocation^{b,k_{\ell}}_{\tau},
  \end{align*}
  for all $\tau \in \neigh(i) \setminus j$.  Taking now the limit as
  $\ell \rightarrow \infty$,
  \begin{align*}
    w_{i,\hat{\tau}} - \allocation^{b,*}_{\hat{\tau}} \ge w_{i,\tau} -
    \allocation^{b,*}_{\tau},
  \end{align*}
  for all $\tau \in \neigh(i) \setminus j$, which contradicts
  $\hat{\tau} \notin \nextbest{i}{j}(\allocation^{b,*})$.
\end{proof}

\end{document}